\documentclass[11pt]{amsart}
\usepackage[usenames]{color}
\usepackage{fullpage}
\usepackage{amscd}
\usepackage{amssymb, latexsym}
\usepackage[all,2cell,ps]{xy}
\usepackage{mathdots}

\theoremstyle{plain}
\newtheorem{thm}{Theorem}[section]
\newtheorem{theorem}[thm]{Theorem}

\newtheorem{corollary}[thm]{Corollary}

\newtheorem{lemma}[thm]{Lemma}

\newtheorem{prop}[thm]{Proposition}
\newtheorem{proposition}[thm]{Proposition}
\newtheorem{ques}[thm]{Question}

\theoremstyle{definition}
\newtheorem{de}[thm]{Definition}

\newtheorem{rem}[thm]{Remark}

\newtheorem{example}[thm]{Example}

\newcommand{\id}{\mathrm{id}}

\newcommand{\Mlt}{\mathop{\mathcal{G}}}
\newcommand{\Aut}{\mathop{\mathrm{Aut}}}

\numberwithin{equation}{section}

\begin{document}

\title{Non-involutive solutions of the Yang-Baxter equation\\ of multipermutation level 2}

\author{Jan Hora}
\author{P\v remysl Jedli\v cka}
\author{Agata Pilitowska}

\address{(J.H., P.J.) Department of Mathematics and Physics, Faculty of Engineering, Czech University of Life Sciences Prague, Kam\'yck\'a 129, 16521 Praha 6, Czech Republic}
\address{(A.P.) Faculty of Mathematics and Information Science, Warsaw University of Technology, Koszykowa 75, 00-662 Warsaw, Poland}

\email{(J.H.) horaj@tf.czu.cz}
\email{(P.J.) jedlickap@tf.czu.cz}
\email{(A.P.) agata.pilitowska@pw.edu.pl}



\date{\today}

\begin{abstract}
We study non-degenerate set-theoretic solutions of the Yang-Baxter equation of multipermutation level $2$ which are not $2$-reductive. We describe an effective way of constructing such solutions using square-free $2$-reductive solutions and two bijections. We present an algorithm how to obtain all such finite solutions, up to isomorphism. Using this algorithm, we enumerate all solutions of multipermutation level $2$ up to size $6$.

\end{abstract}

\maketitle

\section{Introduction}
The Yang-Baxter equation is a fundamental equation occurring in mathematical physics. It appears, for example, in integrable models in statistical mechanics, quantum field theory or Hopf algebras~(see e.g. \cite{Jimbo, K}). Searching for its solutions has been absorbing researchers for many years.

Let us recall that, for a vector space $V$, a {\em solution of the Yang--Baxter equation} is a linear mapping $r:V\otimes V\to V\otimes V$ 
 such that
\begin{align*}
(id\otimes r) (r\otimes id) (id\otimes r)=(r\otimes id) (id\otimes r) (r\otimes id).
\end{align*}

Description of all possible solutions seems to be extremely difficult and therefore
there were some simplifications introduced by Drinfeld in \cite{Dr90}.
Let  $X$ be a basis of the space $V$ and let $\sigma:X^2\to X$ and $\tau: X^2\to X$ be two mappings. We say that $(X,\sigma,\tau)$ is a {\em set-theoretic solution of the Yang--Baxter equation} if
the mapping 
$$x\otimes y \mapsto \sigma(x,y)\otimes \tau(x,y)$$ extends to a solution of the Yang--Baxter
equation. It means that $r\colon X^2\to X^2$, where $r=(\sigma,\tau)$,  
satisfies the \emph{braid relation}:
\begin{equation}\label{eq:braid}
(id\times r)(r\times id)(id\times r)=(r\times id)(id\times r)(r\times id).
\end{equation}

A solution is called {\em non-degenerate} if the mappings $\sigma_x=\sigma(x,\_)$ and $\tau_y=\tau(\_\,,y)$ are bijections,
for all $x,y\in X$.

All solutions we study in this paper are set-theoretic and non-degenerate and we will call them simply \emph{solutions}.
Although researchers usually focus on finite solutions only, in our paper
the set $X$ can be of arbitrary cardinality. 

Important class of solutions is given by multipermutation ones. Gateva-Ivanova studied such solutions in many papers, with special attention to involutive solutions of low multipermutation level. For example, in \cite{GIC12} together with Cameron they gave an equational characterization of square-free involutive solutions of multipermutation level at most $k$. Later on, she investigated in \cite{GI18} in more details  square-free involutive solutions of multipermutation level at most $2$. In \cite{JPZ20} we, together with A. Zamojska-Dzienio, showed that involutive solutions of multipermutation level at most $2$ happen to fall into two classes: $2$-reductive and non $2$-reductive ones. In particular, we gave a combinatorial construction of any such solution. Later, Rump reworked the construction in \cite{Rump22} under the name of \emph{transvection torsor}. 

In \cite[Problem 23]{V19} Vendramin proposed to study non-involutive multipermutation solutions. Following his proposal we characterized in \cite{JP24} a special subclass of such solutions, called $2$-\emph{reductive} ones and,
in this paper we study non $2$-reductive non-involutive solutions of multipermutation level at most $2$. Similarly as in \cite{JPZ20}, we adopt the idea of isotopes from the theory of quasigroups. We show that each solution of multipermutation level at most $2$ may be constructed from a $2$-reductive one. But, contrary to the involutive case, we obtain a characterization in full generality, up to isomorphism. It was possible due to a uniquely determined $2$-reductive square-free isotope of a solution of multipermutation level at most $2$. Our inspiration comes from Rump. In \cite{Rump22} he related each involutive solution of multipermutation level at most $2$ to a unique square-free solution of multipermutation level at most $2$ which in fact is $2$-reductive.   

The paper is organized as follows: in Section~\ref{sec:prelim} we recall basic definitions and properties of solutions. In~Section~\ref{sec:2per} we introduce the definition of $2$-permutational solutions in the general case and we 
prepare several technical lemmas for use in the next sections. Some auxiliary results were obtained using the automated deduction software Prover9 \cite{Prover}.
In~Section~\ref{sec:isotopes} we generalize the concept of isotopes of an arbitrary solution known from the involutive case (see \cite{JPZ20a}). We prove that all isotopes of a given solution have 
some properties in common and we describe conditions under which
	isotopes satisfy certain prescribed properties.
In Section~\ref{sec:isosqfree} we characterize
isotopes of $2$-reductive square-free solutions. 
Key ingredients for these isotopes are diagonal mappings
$U\colon X\to X$, $U(x)= \sigma^{-1}_x(x)$ and $T\colon X\to X$, $T(x)= \tau^{-1}_x(x)$ that turn out to be automorphisms of
$2$-permutational solutions.
In Theorem \ref{th:repres} we show that each $2$-permutational solution is a uniquely defined isotope of a square-free  $2$-reductive one. We also prove the isomorphism Theorem \ref{thm:iso} that determines when two isotopes of a square free  $2$-reductive one represent isomorphic solutions. 
In Section \ref{sec:invol} we restrict the results to involutive solutions. Finally, in Section~\ref{sec:algorithm} we describe the algorithm which 
effectively
constructs all solutions of multipermutation level~2 from square-free $2$-reductive ones. Using this algorithm and the system GAP \cite{gap}, we enumerate all solutions of multipermutation level $2$ up to size $6$. In Proposition \ref{prop:fivecond} we also present a criterion how to recognize whether a $2$-permutational solution is $2$-reductive or not.

\section{Preliminaries}\label{sec:prelim}
If $(X, \sigma,\tau)$ is a solution then directly by the braid relation \eqref{eq:braid} we obtain for $x,y,z\in X$:
\begin{align}
\sigma_x\sigma_y&=\sigma_{\sigma_x(y)}\sigma_{\tau_y(x)}, \label{birack:1}\\
\tau_{\sigma_{\tau_y(x)}(z)}\sigma_x(y)&=\sigma_{\tau_{\sigma_y(z)}(x)}\tau_{z}(y), \label{birack:2}\\
\tau_x\tau_y&=\tau_{\tau_x(y)}\tau_{\sigma_y(x)}. \label{birack:3}
\end{align}

A solution is called {\em bijective} if $r=(\sigma,\tau)$  is a bijection. An example of a bijective solution is given by 
{\em involutive} ones, where $r^2=\mathrm{id}_{X^2}$, i.e. for each $x,y\in X$,
\begin{equation}\label{eq:involutive}
\tau_y(x)=\sigma_{\sigma_x(y)}^{-1}(x)\quad {\rm and} \quad \sigma _x(y)=\tau^{-1}_{\tau_y(x)}(y).
\end{equation} 
Note that, by \cite[Theorem 3.1.]{CJVV22}, each  finite non-degenerate solution is bijective. 
Moreover, it is called
\emph{square-free} if $r(x,x)=(x,x)$, for every $x\in X$.

We say that a solution $(X,\sigma,\tau)$ satisfies  Condition {\bf lri} if, for each $x\in X$, permutations $\sigma_x$ and $\tau_x$ are mutually inverse, i.e. 
\begin{align}
\forall_{x\in X}\quad \sigma_x=\tau_x^{-1}. \tag{lri}
\end{align}
We will say that a solution $(X,\sigma,\tau)$ is \emph{left distributive}, if for every $x,y \in X$:
\begin{align}\label{eq:left}
\sigma_x\sigma_y=\sigma_{\sigma_x(y)}\sigma_x,
\end{align}
and it is \emph{right distributive}, if for every $x,y \in X$:
\begin{align}\label{eq:right}
\tau_x\tau_y=\tau_{\tau_x(y)}\tau_x.
\end{align}
A solution is \emph{distributive} if it is left and right distributive. 
\vskip 3mm
Etingof, Schedler and Soloviev \cite{ESS} introduced, for each solution $(X,\sigma,\tau)$, its \emph{structure group} $G(X,r):=\langle X\mid xy=\sigma_x(y)\tau_y(x)\; \forall x,y\in X\rangle$. A solution is called \emph{injective} if the canonical map $X\to G(X,\sigma,\tau);\; x\mapsto x$ is injective. Involutive solutions are always injective.

\subsection{Retraction solution}
Let $(X,\sigma,\tau)$ be a solution. An equivalence relation $\mathord{\asymp}\subseteq X\times X$ such that for $x_1,x_2,y_1,y_2\in X$ 
\begin{align}\label{congr}
&x_1\asymp x_2\;\; {\rm and} \;\; y_1\asymp y_2\quad \Rightarrow\quad \sigma^{\varepsilon}_{x_1}(y_1)\asymp \sigma^{\varepsilon}_{x_2}(y_2)\quad {\rm and}\quad \tau^{\varepsilon}_{x_1}(y_1)\asymp \tau^{\varepsilon}_{x_2}(y_2),
\end{align}
where $\varepsilon\in \{-1,1\}$, is called a \emph{congruence} of the solution $(X,\sigma,\tau)$.
A congruence induces a quotient solution on its classes.

In \cite{ESS} Etingof, Schedler and Soloviev introduced, for each involutive solution $(X,\sigma,\tau)$, the equivalence relation $\sim$ on the set $X$: for each $x,y\in X$
\begin{align}\label{rel:sim}
x\sim y\quad \Leftrightarrow\quad \sigma_x=\sigma_y
\end{align}
and they showed that 
$\sim$ is a congruence of the solution. 
In the case of non-involutive solution $(X,\sigma,\tau)$,
the equivalence relation $\sim$
need not to 
be a congruence. But it is so if the solution is left distributive (see \cite[Theorem 3.4]{JPZ20}). 
Analogously to \eqref{rel:sim}, we can define the symmetrical relation
\begin{equation}
x\backsim y \quad \Leftrightarrow\quad \tau_x=\tau_y
\end{equation}
and this relation induces a solution on the quotient set $X^{\backsim}$ of every right distributive solution. 
If a 
solution is involutive then $x\sim y$ if and only if $x\backsim y$ \cite[Proposition 2.2]{ESS}.

The intersection of the two relations here defined is
the relation
\begin{equation}\label{eq:retraction}
 x\approx y \quad \Leftrightarrow \quad {x\sim y} \wedge x\backsim y
 \quad \Leftrightarrow \quad {\sigma_x=\sigma_y} \wedge {\tau_x=\tau_y}.
\end{equation}

Lebed and Vendramin showed in \cite{LV} that  the relation $\approx$ is a congruence of injective solutions. 
In \cite{JPZ19} the two of the authors together with Zamojska-Dzienio proved that the relation $\mathrel{\approx}$ induces a solution on the quotient set $X^{\mathrel{\approx}}$  for any solution $(X,\sigma,\tau)$.  A substantially shorter proof for bijective solutions has recently appeared in \cite{CJKAV}.
\begin{de}\label{ret}
Let $(X,\sigma,\tau)$ be a solution. The quotient solution $\mathrm{Ret}(X,\sigma,\tau):=(X^{\approx},\sigma,\tau)$  with $\sigma_{x^{\approx}}(y^{\approx})=\sigma_x(y)^{\approx}$ and $\tau_{y^{\approx}}(x^{\approx})=\tau_y(x)^{\approx}$, for $x^{\approx},y^{\approx}\in X^{\approx}$  and $x\in x^{\approx},\; y\in y^{\approx}$, is called the \emph{retraction} solution of $(X,\sigma,\tau)$. 
We say that
$(X,\sigma,\tau)$ has {\em multipermutation level $2$} if ${\rm Ret}^2(X,\sigma,\tau):={\rm Ret}({\rm Ret}(X,\sigma,\tau))$
has one element only. 
We say that a solution $(X,\sigma,\tau)$ is \emph{ire\-trac\-table} if ${\rm Ret}(X,\sigma,\tau)=(X,\sigma,\tau)$, i.e.
$\approx$ is the trivial relation.

\end{de} 
A solution $(X,\sigma,\tau)$ is called \emph{permutational}, if it has multipermutation level $1$, i.e. for every $x,y\in X$, $\sigma_x=\sigma_y$  and $\tau_x=\tau_y$. It is a \emph{projection} (or \emph{trivial}) solution if for every $x\in X$, $\sigma_x=\tau_x=\id$.

For a non-empty set $X$ and two bijections  $f,g\colon X\to X$ such that $fg=gf$, the  permutational solution $(X,\sigma,\tau)$ with $\sigma_x=f$ and $\tau_y=g$, for each $x,y\in X$, is distributive [Lyubashenko, see \cite{Dr90}].

\subsection{Inverse solution}

If a solution $(X,\sigma,\tau)$ is bijective then there is $r^{-1}\colon X^2\to X^2$ such that $rr^{-1}=r^{-1}r=\id$. It is also true that $(X,r^{-1})$ is a solution. We will call it the \emph{inverse solution} to $(X,\sigma,\tau)$. 

Let for $x\in X$, $\hat{\sigma}_x,\hat{\tau}_x\colon X\to X$ be such that $r^{-1}(x,y)=(\hat{\sigma}_x(y),\hat{\tau}_y(x))$. Clearly, we have that for $x,y\in X$:
\begin{align*}
&(x,y)=rr^{-1}(x,y)=r(\hat{\sigma}_x(y),\hat{\tau}_y(x))=(\sigma_{\hat{\sigma}_x(y)}\hat{\tau}_y(x),\tau_{\hat{\tau}_y(x)}\hat{\sigma}_x(y)), \quad {\rm and}\\
&(x,y)=r^{-1}r(x,y)=r^{-1}(\sigma_x(y),\tau_y(x))=(\hat{\sigma}_{\sigma_x(y)}\tau_y(x),\hat{\tau}_{\tau_y(x)}\sigma_x(y)).
\end{align*}
Hence 
\begin{align}
&\sigma_{\hat{\sigma}_x(y)}\hat{\tau}_y(x)=x\quad \Rightarrow\quad \hat{\tau}_y(x)=\sigma^{-1}_{\hat{\sigma}_x(y)}(x)\quad \Rightarrow\quad \sigma^{-1}_y(x)=\hat{\tau}_{\hat{\sigma}^{-1}_x(y)}(x)\label{rr:1}\\
&\tau_{\hat{\tau}_y(x)}\hat{\sigma}_x(y)=y\quad \Rightarrow\quad \hat{\sigma}_x(y)=\tau^{-1}_{\hat{\tau}_y(x)}(y)\quad \Rightarrow\quad \tau^{-1}_x(y)=\hat{\sigma}_{\hat{\tau}^{-1}_y(x)}(y)\label{rr:2}\\
&\hat{\sigma}_{\sigma_x(y)}\tau_y(x)=x\quad \Rightarrow\quad \tau_y(x)=\hat{\sigma}^{-1}_{\sigma_x(y)}(x)\quad \Rightarrow\quad \hat{\sigma}^{-1}_y(x)=\tau_{\sigma^{-1}_x(y)}(x)\label{rr:3}\\
&\hat{\tau}_{\tau_y(x)}\sigma_x(y)=y\quad \Rightarrow\quad \sigma_x(y)=\hat{\tau}^{-1}_{\tau_y(x)}(y)\quad \Rightarrow\quad \hat{\tau}^{-1}_x(y)=\sigma_{\tau^{-1}_y(x)}(y).\label{rr:4}
\end{align}
By definition, in an involutive solution we always have $\hat{\sigma}_x=\sigma_x$ and $\hat{\tau}_x=\tau_x$. 





\vskip 2mm




\subsection{Displacement group}
In the case of involutive solution $(X,\sigma,\tau)$, the permutation group of a solution is defined as the subgroup $$\Mlt(X)=\langle\sigma_x:x\in X\rangle$$ of the symmetric group $S(X)$ generated by all translations $\sigma_x$, with $x\in X$. 

Since, for each $x,y\in X$, $\tau_y(x)=\sigma_{\sigma_x(y)}^{-1}(x)$, the group $\Mlt(X)$ can be equivalently defined by permutations $\tau_x$. 
%
For the non-involutive case
Bachiller defined in \cite[Definition 3.10]{B18} the permutation group of the solution as some subgroup of the product $S(X)\times S(X)$. Ced\'o et al. showed in \cite{CJKAV} that, for bijective solutions, the permutation group defined by Bachiller may be described equivalently as the group generated by all pairs $(\sigma_x,\hat{\sigma}_x)$, for $x\in X$. Moreover, in \cite[Lemma 1.3]{CJKAV} they proved that such group is isomorphic to the group generated by all pairs of the form $(\sigma_x,\tau^{-1}_x)$, with $x\in X$. We will adopt this fact as a definition. 

\begin{de}
The permutation group $\Mlt(X)$ of the solution $(X,\sigma,\tau)$ is a subgroup of $S(X)\times S(X)$ defined in the following way:
$$\Mlt(X)=\langle(\sigma_x,\tau_x^{-1}):x\in X\rangle.$$
\end{de}

\begin{de}
The displacement group of the solution $(X,\sigma,\tau)$, denoted by Dis$(X)$, is the group 
$$\langle(\sigma_x\sigma^{-1}_y,\tau_x^{-1}\tau_y):x,y\in X\rangle.$$
\end{de}
The displacement group 
is a normal subgroup of the permutation group \cite[Lemma 3.2]{CJKV23} and it
plays important role not only in the theory of racks and quandles, but also in study of solutions of the Yang-Baxter equation (see \cite{JP23}).

\subsection{Isomorphism}
Let us recall that a bijection $\Phi\colon X\to X'$ is an \emph{isomorphism} of two solutions $(X,\sigma,\tau)$ and $(X',\sigma',\tau')$ if, for each $x\in X$,
\begin{equation}\label{isomorphism}
\Phi\sigma_x=\sigma'_{\Phi(x)}\Phi\quad {\rm and}\quad \Phi\tau_x=\tau'_{\Phi(x)}\Phi.
\end{equation}
An isomorphism is called {\em automorphism} of a solution if $(X,\sigma,\tau)=(X',\sigma',\tau')$.
	The group of automorphisms of a solution $(X,\sigma,\tau)$
	is denoted by $\Aut((X,\sigma,\tau))$.

\section{$2$-permutational solutions}\label{sec:2per}

This section is devoted to a syntactic study of solutions of multipermutation level~2. A significant subclass are so called $2$-reductive solutions.

\begin{de}[\cite{JP23a}]
A solution $(X,\sigma, \tau)$ is $2$-\emph{reductive} if, for every $x,y\in X$:
\begin{align}
& \sigma_{\sigma_x(y)}=\sigma_y, \label{eq:red1}\\
& \tau_{\tau_x(y)}=\tau_y, \label{eq:red2}\\
& \sigma_{\tau_x(y)}=\sigma_y, \label{eq:red3}\\
& \tau_{\sigma_x(y)}=\tau_y. \label{eq:red4}
\end{align}
\end{de}

For $2$-reductive solution we also have:
\begin{align}\label{eq:more2red}
& \sigma_{\sigma^{-1}_x(y)}=\sigma_y,\quad\tau_{\tau^{-1}_x(y)}=\tau_y , \quad \tau_{\sigma^{-1}_x(y)}=\tau_y\quad{\rm and}\quad \sigma_{\tau^{-1}_x(y)}=\sigma_y.
\end{align}

The structure of $2$-reductive solutions was described in~\cite{JP23a}. They can be obtained by a combinatorial construction that uses a series of abelian groups and two matrices. Moreover, all the permutations $\sigma_x$ and $\tau_y$ are automorphisms of the solution $(X,\sigma,\tau)$.

In~\cite{JP24} the authors presented an equational characterization of
multipermutation solution. In the case of level~$2$ it uses the following notion:

\begin{de}
A solution $(X,\sigma, \tau)$ is $2$-\emph{permutational} if, for every $x,y,z\in X$:
\begin{align}
& \sigma_{\sigma_x(z)}=\sigma_{\sigma_y(z)}, \label{eq:per1}\\
& \tau_{\tau_x(z)}=\tau_{\tau_y(z)}, \label{eq:per2}\\
& \sigma_{\tau_x(z)}=\sigma_{\tau_y(z)}, \label{eq:per3}\\
& \tau_{\sigma_x(z)}=\tau_{\sigma_y(z)}. \label{eq:per4}
\end{align}
\end{de}

\begin{proposition}[{\cite[Theorem 4.4]{JP24}}]
	\label{cor:multi2per2}
	$(X,\sigma,\tau)$ is a multipermutation solution of level at most $2$ if and only if it is $2$-permutational.
%
\end{proposition}

Note that in $2$-permutational solution $(X,\sigma, \tau)$, for each $x,y,z,a,b\in X$,
\begin{align*}
&\tau_{\sigma_{\tau_y(x)}(z)}\sigma_x(y)\stackrel{\eqref{birack:2}}=\sigma_{\tau_{\sigma_y(z)}(x)}\tau_{z}(y) \quad 
\stackrel{\eqref{eq:per3},\eqref{eq:per4}}\Rightarrow\quad 
\tau_{\sigma_{a}(z)}\sigma_x(y)=\sigma_{\tau_{b}(x)}\tau_{z}(y).
\end{align*}
This means that Condition \eqref{birack:2} reduces to 
\begin{align}\label{birack2per}
\tau_{\sigma_{a}(z)}\sigma_x=\sigma_{\tau_{b}(x)}\tau_{z},
\end{align}
for $x,z,a,b\in X$.

Moreover, according to \cite[Theorem 3.4]{JP24}, every
$2$-permutational solutions also satisfies
\begin{align}\label{eq:more2per}
&\tau_{\sigma^{-1}_x(y)}=\tau_{\sigma^{-1}_z(y)}\quad{\rm and}\quad \sigma_{\tau^{-1}_x(y)}=\sigma_{\tau^{-1}_z(y)},
\end{align}
 for $x,y,z\in X$.
 

Clearly, each $2$-reductive solution is $2$-permutational  and each square-free $2$-permutational solution is $2$-reductive.





\begin{lemma}\label{lm:basic2per}
Let $(X,\sigma,\tau)$ be a $2$-permutational solution. Then, for every $x,y,z\in X$, we have:
\begin{align}
&\sigma_{\tau_x^{-1}(y)}=\sigma_{\sigma_z(y)},\label{eq:tauinverse1}\\
&\tau_{\sigma_x^{-1}(y)}=\tau_{\tau_z(y)},\label{eq:tauinverse5}\\
&\sigma_{\tau_x\sigma_z(y)}=\sigma_{\sigma_z\tau_x(y)}=\sigma_y,\label{eq:tauinverse3}\\
&\tau_{\sigma_x\tau_z(y)}=\tau_{\tau_z\sigma_x(y)}=\tau_y,\label{eq:tausigmatau}\\
&\sigma_{\sigma^{-1}_y(x)}=\sigma_{\tau_z(x)},\label{eq:sigmasigmainverse}\\
&\tau_{\tau^{-1}_y(x)}=\tau_{\sigma_z(x)}.\label{eq:tautauinverse}
\end{align}
\end{lemma}
\begin{proof}
Let $(X,\sigma,\tau)$ be a $2$-permutational solution and $x,y,z,t\in X$. Then\\
\eqref{eq:tauinverse1}:
\begin{align*}
&\sigma_x\sigma_y\stackrel{\eqref{birack:1}}=\sigma_{\sigma_x(y)}\sigma_{\tau_y(x)}\stackrel{\eqref{eq:per1}}=\sigma_{\sigma_z(y)}\sigma_{\tau_y(x)}
 \; \Rightarrow\; \sigma^{-1}_{\sigma_z(y)}\sigma_x=\sigma_{\tau_y(x)}\sigma_y^{-1}\;
\stackrel{x\mapsto \tau^{-1}_y(y)}\Rightarrow\\
&\sigma^{-1}_{\sigma_z(y)}\sigma_{\tau^{-1}_y(y)}=\sigma_y\sigma_y^{-1}\; \Leftrightarrow\;
\sigma_{\tau^{-1}_y(y)}=\sigma_{\sigma_z(y)}\stackrel{\eqref{eq:more2per}}\Rightarrow\;
\sigma_{\tau^{-1}_x(y)}=\sigma_{\sigma_z(y)}.
\end{align*}
\eqref{eq:tauinverse5}:\\
Analogously as for \eqref{eq:tauinverse1}, replacing \eqref{birack:1} by \eqref{birack:3} and \eqref{eq:per1} by \eqref{eq:per4}.
\\
\eqref{eq:tauinverse3}:

\begin{align*}
&\sigma_t\sigma_y\stackrel{\eqref{birack:1}}=\sigma_{\sigma_t(y)}\sigma_{\tau_y(t)}\stackrel{\eqref{eq:per1},\eqref{eq:per3}}=\sigma_{\sigma_z(y)}\sigma_{\tau_x(t)}  \;
\stackrel{t\mapsto \sigma_z(y)}\Rightarrow\;
\sigma_{ \sigma_z(y)}\sigma_y=\sigma_{\sigma_z(y)}\sigma_{\tau_x\sigma_z(y)} 
\quad\Rightarrow\quad
\sigma_y=\sigma_{\tau_x\sigma_z(y)},
\end{align*}
and
\begin{align*}
&\sigma_{\tau_x^{-1}(y)}\stackrel{\eqref{eq:tauinverse1}}=\sigma_{\sigma_z(y)}\;
\stackrel{y\mapsto \tau_x(y)}\Rightarrow\;\sigma_y=\sigma_{\sigma_z\tau_x(y)}.
\end{align*}
\\
\eqref{eq:tausigmatau}:
\begin{align*}
&\tau_t\tau_y\stackrel{\eqref{birack:3}}=\tau_{\tau_t(y)}\tau_{\sigma_y(t)}\stackrel{\eqref{eq:per2},\eqref{eq:per4}}=\tau_{\tau_z(y)}\tau_{\sigma_x(t)}  \;
\stackrel{t\mapsto \tau_z(y)}\Rightarrow\;
\tau_{ \tau_z(y)}\tau_y=\tau_{\tau_z(y)}\tau_{\sigma_x\tau_z(y)} 
\quad\Rightarrow\quad
\tau_y=\tau_{\sigma_x\tau_z(y)},
\end{align*}
and
\begin{align*}
&\tau_{\sigma_x^{-1}(y)}\stackrel{\eqref{eq:tauinverse5}}=\tau_{\tau_z(y)}\;
\stackrel{y\mapsto \sigma_x(y)}\Rightarrow\;\tau_y=\tau_{\tau_z\sigma_x(y)}.
\end{align*}
\eqref{eq:sigmasigmainverse}:
\begin{align*}
&\sigma_{\sigma^{-1}_y(x)}\stackrel{\eqref{eq:tauinverse3}}=\sigma_{\tau_z\sigma_y\sigma_y^{-1}(x)}=\sigma_{\tau_z(x)}.
\end{align*}
\eqref{eq:tautauinverse}:
\begin{align*}
&\tau_{\tau^{-1}_y(x)}\stackrel{\eqref{eq:tausigmatau}}=\tau_{\sigma_z\tau_y\tau_y^{-1}(x)}=\tau_{\sigma_z(x)}. \qedhere
\end{align*}
\end{proof}


Recall, for a distributive solution $(X,\sigma,\tau)$, we have (see also \cite[Lemma 2.8]{JPZ20})
\begin{align*}
&\sigma_{\sigma_x(y)}\sigma_x\stackrel{\eqref{eq:left}}=\sigma_x\sigma_y\stackrel{\eqref{birack:1}}=\sigma_{\sigma_x(y)}\sigma_{\tau_y(x)}\quad \Rightarrow\quad \sigma_x=\sigma_{\tau_y(x)}.
\end{align*}
And similarly, by \eqref{eq:right} and \eqref{birack:3}, 
$$\tau_x=\tau_{\sigma_y(x)}.$$

%

\begin{lemma}\label{lm:perdis}
Each $2$-permutational and distributive solution is $2$-reductive.
\end{lemma}
\begin{proof}
Let $(X,\sigma,\tau)$ be a $2$-permutational and distributive solution. 
For each $x,y,z\in X$, we have
\begin{align*}
&\sigma_x\sigma_y\sigma_x^{-1}\stackrel{\eqref{eq:left}}=\sigma_{\sigma_x(y)}
\stackrel{\eqref{eq:per1}}=
\sigma_{\sigma_z(y)}\stackrel{\eqref{eq:left}}=
\sigma_z\sigma_y\sigma_z^{-1}\quad \Rightarrow\\
&\sigma_x\sigma_y\sigma_x^{-1}=\sigma_z\sigma_y\sigma_z^{-1}\quad
\stackrel{y\mapsto x}\Rightarrow\quad \sigma_x=\sigma_z\sigma_x\sigma_z^{-1}\stackrel{\eqref{eq:left}}=\sigma_{\sigma_z(x)}.
\end{align*}
In the similar way we can show that $\tau_x=\tau_{\tau_z(x)}$. Because, for every distributive solution, $\sigma_x=\sigma_{\tau_y(x)}$ and $\tau_x=\tau_{\sigma_y(x)}$, it finishes the proof.
\end{proof}

\begin{lemma}
Let $(X,\sigma,\tau)$ be a $2$-permutational solution. Then, for every $x,y,z,a,b\in X$, we have:
\begin{align}
&\sigma^{-1}_{y}\sigma_x=\sigma_{\tau_z(x)}\sigma_{\tau_b(y)}^{-1},\label{eq:tauinverse2}\\
&\tau_x\tau_y^{-1}=\tau^{-1}_{\sigma_b^{-1}(y)}\tau_{\sigma_z^{-1}(x)}\label{eq:tauinverse6}\\
&\sigma_x\tau_z=\tau_{\sigma_a(z)}\sigma_{\sigma_b(x)},\label{eq:tauinverse4}\\
&\sigma^{-1}_x\tau_z=\tau_{\sigma^{-1}_a(z)}\sigma^{-1}_{\sigma_b(x)},\label{eq:41}\\
&\sigma_{\sigma_x(y)}\sigma_{\sigma_a(z)}=\sigma_{\sigma_x\sigma_a(z)}\sigma_y,\label{eq:44}\\
&\tau_{\tau_x(y)}\tau_{\tau_a(z)}=\tau_{\tau_x\tau_a(z)}\tau_y.\label{eq:45}
\end{align}

\end{lemma}
\begin{proof}
Let $(X,\sigma,\tau)$ be a $2$-permutational solution and $x,y,z,a,b\in X$. Then
\begin{align*}
&\eqref{eq:tauinverse2}:&&\sigma_x\sigma_y\stackrel{\eqref{birack:1}}=\sigma_{\sigma_z(y)}\sigma_{\tau_y(x)} \; \Rightarrow\; \sigma^{-1}_{\sigma_z(y)}\sigma_x=\sigma_{\tau_y(x)}\sigma_y^{-1}\;
\stackrel{\eqref{eq:tauinverse1}}\Rightarrow\\
&&&\sigma^{-1}_{\tau^{-1}_b(y)}\sigma_x=\sigma_{\tau_y(x)}\sigma_y^{-1}\stackrel{\eqref{eq:per1}}=\sigma_{\tau_z(x)}\sigma_y^{-1}\;\stackrel{y\mapsto \tau_b(y)}\Rightarrow\;
\sigma^{-1}_{y}\sigma_x=\sigma_{\tau_z(x)}\sigma_{\tau_b(y)}^{-1}.\\
&\eqref{eq:tauinverse6}:&
&\tau_x\tau_y\stackrel{\eqref{birack:3}}=\tau_{\tau_x(y)}\tau_{\sigma_y(x)}\;\Rightarrow\;\tau^{-1}_{\tau_x(y)}\tau_x=\tau_{\sigma_y(x)}\tau_y^{-1}\; \stackrel{\eqref{eq:tauinverse5}}\Rightarrow\\
&&&\tau^{-1}_{\sigma^{-1}_b(y)}\tau_x=\tau_{\sigma_y(x)}\tau_y^{-1}\stackrel{\eqref{eq:per2}}=\tau_{\sigma_z(x)}\tau_y^{-1}\;\stackrel{x\mapsto \sigma^{-1}_z(x)}\Rightarrow\;
\tau^{-1}_{\sigma^{-1}_b(y)}\tau_{\sigma_z^{-1}(x)}=\tau_x\tau_y^{-1}.\\
&\eqref{eq:tauinverse4}:&
&\tau_{\sigma_{a}(z)}\sigma_x\stackrel{\eqref{birack2per}}=\sigma_{\tau_{b}(x)}\tau_{z}\;\stackrel{x\mapsto \sigma_b(x)}
\Rightarrow\; \tau_{\sigma_{a}(z)}\sigma_{\sigma_b(x)}=\sigma_{\tau_{b}\sigma_b(x)}\tau_{z} \stackrel{\eqref{eq:tauinverse3}}= \sigma_x\tau_z.\\
&\eqref{eq:41}:&
&\tau_z\sigma^{-1}_{\sigma_b(x)}\stackrel{\eqref{eq:tauinverse4}}=\sigma^{-1}_x\tau_{\sigma_a(z)}\;\stackrel{z\mapsto \sigma^{-1}_a(z)}
\Rightarrow\;\tau_{\sigma^{-1}_a(z)}\sigma^{-1}_{\sigma_b(x)}=\sigma^{-1}_x\tau_{\sigma_a\sigma^{-1}_a(z)}=\sigma^{-1}_x\tau_z.\\
&\eqref{eq:44}:&
&\sigma_{\sigma_x(y)}\sigma_{\sigma_a(z)}\stackrel{\eqref{birack:1}}=
\sigma_{\sigma_{\sigma_x(y)}\sigma_a(z)}
\sigma_{\tau_{\sigma_a(z)}\sigma_x(y)}
\stackrel{\eqref{eq:per1}+\eqref{eq:tauinverse3}}=\sigma_{\sigma_x\sigma_a(z)}\sigma_y,
\end{align*}
and analogously we obtain \eqref{eq:45}.
\end{proof}

\begin{de}
	A structure with several binary operations is called {\em entropic} if
	\[(x*_1y)*_2(z*_1 u)=(x*_2z)*_1(y*_2 u),\]
	for all binary operations $*_1$, $*_2$ and all $x,y,z,u\in X$.
\end{de}

In \cite{JPZ20a}, together with Zamojska-Dzienio, two of the authors showed that an involutive solution is $2$-permutational if and only if it is entropic. The implication in one direction holds also for arbitrary solutions.
\begin{lemma}
Each $2$-permutational solution is entropic.
\end{lemma}
\begin{proof}
Let $(X,\sigma,\tau)$ be a $2$-permutational solution and let $x,y,z,u\in X$. Then
\begin{align*}
&\sigma_{\sigma_x(y)}\sigma_{\sigma_a(z)}\stackrel{\eqref{eq:44}}=
\sigma_{\sigma_x\sigma_a(z)}\sigma_y
\stackrel{z\mapsto \sigma^{-1}_a(z)}
\Rightarrow\; \sigma_{\sigma_x(y)}\sigma_{z}=
\sigma_{\sigma_x(z)}\sigma_y.
\end{align*}
Similarly we can show that $\tau_{\tau_x(y)}\tau_{z}=
\tau_{\tau_x(z)}\tau_y$. Further,
\begin{align*}
&\tau_{\sigma_z(u)}\sigma_x(y)
\stackrel{\eqref{eq:per4}}=
\tau_{\sigma_{\tau_y(x)}(u)}\sigma_x(y)
\stackrel{\eqref{birack:2}}=
\sigma_{\tau_{\sigma_y(u)}(x)}\tau_u(y)
\stackrel{\eqref{eq:per3}}=
\sigma_{\tau_z(x)}\tau_u(y),
\end{align*}
which finishes the proof.
\end{proof}

The converse implication is not true in general since a solution $(X,\sigma,\id)$, where $(X,\sigma)$ is a latin quandle (i.e. idempotent,  distributive quasigroup), is distributive and entropic but not retractable.

\begin{lemma}
Let $(X,\sigma,\tau)$ be a $2$-permutational solution. Then, for every $x,y,a,b\in X$, we have:
\begin{align}
&\sigma_a\sigma_y^{-1}\sigma_x=\sigma_x\sigma_y^{-1}\sigma_a,\label{eq:com1}\\
&\tau_a\tau_y^{-1}\tau_x=\tau_x\tau_y^{-1}\tau_a,\label{eq:com12}\\
&\sigma_a\sigma_y^{-1}\tau_b\tau_y^{-1}=\tau_b\tau_y^{-1}\sigma_a\sigma_y^{-1}.\label{eq:com2}
\end{align}

\end{lemma}
\begin{proof}
Let $(X,\sigma,\tau)$ be a $2$-permutational solution and $x,y,z,t,w,a,b\in X$. Then\\
\eqref{eq:com1}:
\begin{align*}
&\sigma_{\tau_z(x)}\stackrel{\eqref{birack:1}}=\sigma^{-1}_{\sigma_x(z)}\sigma_x\sigma_z\;\Rightarrow\;\sigma^{-1}_{\tau_z(x)}=\sigma^{-1}_z\sigma^{-1}_x\sigma_{\sigma_x(z)}\;\Rightarrow\\
&\sigma^{-1}_{y}\sigma_x
\stackrel{\eqref{eq:tauinverse2}}=\sigma_{\tau_z(x)}\sigma_{\tau_b(y)}^{-1}=\sigma^{-1}_{\sigma_x(z)}\sigma_x\sigma_z
\sigma^{-1}_b\sigma^{-1}_y\sigma_{\sigma_y(b)}
\;\stackrel{z=b=\sigma^{-1}_x(a)}\Rightarrow\\
&\sigma^{-1}_{y}\sigma_x=
\sigma^{-1}_{a}\sigma_x\sigma^{-1}_y\sigma_{\sigma_y\sigma^{-1}_x(a)}\stackrel{\eqref{eq:per1}}=\sigma^{-1}_{a}\sigma_x\sigma^{-1}_y\sigma_{\sigma_x\sigma^{-1}_x(a)}=\sigma^{-1}_{a}\sigma_x\sigma^{-1}_y\sigma_{a}\;\Rightarrow\\
&\sigma_a\sigma^{-1}_{y}\sigma_x=
\sigma_x\sigma^{-1}_y\sigma_{a}.
\end{align*}
\eqref{eq:com12}: The proof goes analogously as for \eqref{eq:com1}.\\
\eqref{eq:com2}:
\begin{align*}
&\sigma_a\sigma_y^{-1}\tau_b\tau_y^{-1}\stackrel{\eqref{eq:41}}=\sigma_a\tau_{\sigma^{-1}_t(b)}\sigma^{-1}_{\sigma_w(y)}\tau_y^{-1}\stackrel{\eqref{eq:tauinverse4}}=
\sigma_a\tau_{\sigma^{-1}_t(b)}
\tau^{-1}_{\sigma^{-1}_z(y)}\sigma^{-1}_{\sigma^{-1}_w\sigma_w(y)}=\\
&\sigma_a\tau_{\sigma^{-1}_t(b)}
\tau^{-1}_{\sigma^{-1}_z(y)}\sigma^{-1}_{y}\stackrel{\eqref{eq:tauinverse4}}=
\tau_{\sigma_z\sigma_t^{-1}(b)}\sigma_{\sigma_w(a)}
\tau^{-1}_{\sigma^{-1}_z(y)}\sigma^{-1}_{y}\stackrel{\eqref{eq:per1}}=\tau_{b}\sigma_{\sigma_b(a)}
\tau^{-1}_{\sigma^{-1}_z(y)}\sigma^{-1}_{y}\stackrel{\eqref{eq:41}}=\\
&
=\tau_{b}
\tau^{-1}_{y}
\sigma_{a}
\sigma^{-1}_{y}.\qedhere
\end{align*}
\end{proof}



\begin{prop}\label{thm:dis2perabel}
The displacement group of a $2$-permutational solution is an abelian normal subgroup of $\Mlt(X)$.
\end{prop}
\begin{proof}
Commutativity of a displacement group of a $2$-permutational solution follows directly by \eqref{eq:com1} and \eqref{eq:com12}. Moreover, by \cite[Lemma 3.2]{CJKV23} the displacement group 
is a normal subgroup of $\Mlt(X)$.
\end{proof}

By \eqref{rr:1} and \eqref{rr:2}, it is easy to note that if a solution $(X,\sigma,\tau)$ is permutational then it is bijective and the inverse solution $(X,\hat\sigma,\hat\tau)$ is permutational too. The same is also true for bijective solutions of multipermutation level $2$.

\begin{lemma}
The inverse solution to a bijective solution of multipermutation level~$2$ is a solution of multipermutation level~$2$.
\end{lemma}

\begin{proof}
Let $(X,\sigma,\tau)$ be a bijective solution of multipermutation level at most $2$ and let  $(X,\hat\sigma,\hat\tau)$ be the inverse solution. Then for $x,y,z,t,u,a,b\in X$ we have
\begin{align*}
&\hat\tau^{-1}_x(y)\stackrel{\eqref{rr:4}}=\sigma_{\tau_y^{-1}(x)}(y)\stackrel{\eqref{eq:tauinverse1}}=\sigma_{\sigma_z(x)}(y)\quad \Rightarrow\quad \hat\tau_x=\sigma^{-1}_{\sigma_z(x)}\quad \Rightarrow\\
&\hat\tau_{\hat\tau_a(y)}=\sigma^{-1}_{\sigma_z\hat\tau_a(y)}\stackrel{\eqref{eq:per1}}=\sigma^{-1}_{\sigma_{\hat\sigma_y(a)}\hat\tau_a(y)}\stackrel{\eqref{rr:1}}=\sigma_y^{-1}\quad \Rightarrow\quad \hat\tau_{\hat\tau_a(y)}=\sigma_y^{-1}=\hat\tau_{\hat\tau_b(y)};\\
&\hat\sigma_y^{-1}(x)\stackrel{\eqref{rr:3}}=\tau_{\sigma^{-1}_x(y)}(x)\stackrel{\eqref{eq:tauinverse5}}=\tau_{\tau_z(y)}(x)\quad \Rightarrow\quad \hat\sigma_y=\tau^{-1}_{\tau_z(y)}\quad \Rightarrow\\
&\hat\sigma_{\hat\sigma_a(y)}=\tau^{-1}_{\tau_z\hat\sigma_a(y)}\stackrel{\eqref{eq:per4}}=\tau^{-1}_{\tau_{\hat\tau_y(a)}\hat\sigma_a(y)}\stackrel{\eqref{rr:2}}=\tau^{-1}_y\quad \Rightarrow\quad \hat\sigma_{\hat\sigma_a(y)}=\tau^{-1}_y=\hat\sigma_{\hat\sigma_b(y)};\\
&\hat\tau_x=\sigma^{-1}_{\sigma_z(x)}\quad \stackrel{x\mapsto \hat\sigma_a(y)}\Rightarrow\quad \hat\tau_{\hat\sigma_a(y)}=\sigma^{-1}_{\sigma_z\hat\sigma_a(y)}\stackrel{\eqref{rr:2}}=
\sigma^{-1}_{\sigma_z\tau^{-1}_{\hat\tau_y(a)}(y)}\stackrel{\eqref{eq:41}}=\\
&\sigma^{-1}_{\tau^{-1}_{\sigma_t\hat\tau_y(a)}\sigma_{\sigma^{-1}_s(z)}(y)}\stackrel{\eqref{eq:per1}+\eqref{eq:tauinverse1}}=\sigma^{-1}_{\tau^{-1}_u\sigma_{\sigma^{-1}_s(z)}(y)}=\sigma^{-1}_{\tau^{-1}_{\sigma_t\hat\tau_y(b)}\sigma_{\sigma^{-1}_s(z)}(y)}=\hat\tau_{\hat\sigma_b(y)};\\
&\hat\sigma_y=\tau^{-1}_{\tau_z(y)}\quad \stackrel{y\mapsto \hat\tau_a(y)}\Rightarrow\quad \hat\sigma_{\hat\tau_a(y)}=\tau^{-1}_{\tau_z\hat\tau_a(y)}\stackrel{\eqref{rr:1}}=\hat\sigma_{\hat\tau_a(y)}=\tau^{-1}_{\tau_z\sigma^{-1}_{\hat\sigma_y(a)}(y)}\stackrel{\eqref{eq:41}}=\\
&\tau^{-1}_{\sigma^{-1}_{\sigma^{-1}_t\hat\sigma_y(a)}\tau_{\sigma_s(z)}(y)}\stackrel{\eqref{eq:tauinverse5}+\eqref{eq:per4}}=\tau^{-1}_{\sigma^{-1}_u\tau_{\sigma_s(z)}(y)}=\tau^{-1}_{\sigma^{-1}_{\sigma^{-1}_t\hat\sigma_y(b)}\tau_{\sigma_s(z)}(y)}=\hat\sigma_{\hat\tau_b(y)}.
\end{align*}
As a conclusion, according to Proposition \ref{cor:multi2per2},
	the inverse solution is of multipermutation level at most~$2$.
	If $(X,\hat\sigma,\hat\tau)$ is of multipermutation level~$1$ then $(X,\sigma,\tau)$ is of multipermutation level~$1$, which finishes the proof.
\end{proof}

\begin{ques}
	Is it true that the inverse solution to a solution of multipermutation level~$k$ is of multipermutation level~$k$?
\end{ques}

\section{Isotopes of solutions}\label{sec:isotopes}
In this section we will show that each $2$-permutational solution originates from some $2$-reductive one. 
In the full generality, the concept of \emph{isotopy} 
used in the theory of quasigroups 
consists of permuting rows, columns and symbols of the multiplication table, each of them using a different permutation.
However, permuting the symbols is not important,
up to isomorphism, unless we consider special constants.
In our context we do not have important constants and therefore we do not need to consider the permutation of symbols. Moreover, we are dealing with one-sided quasigroups and therefore it makes sense to
permute either the columns or the rows only.
This narrowed concept of isotopy was
successfully applied in the involutive case already (see \cite{JPZ20a})
giving us a tool to connect involutive solutions of multipermutation level~2 with involutive 2-reductive solutions. In this section we show that this concept generalizes straightforwardly.

Throughout all the section, the permutations of a 2-reductive solution are denoted by $L_x$ and $\mathbf{R}_x$, rather than $\sigma_x$ and $\tau_x$,
in order to have a clear distinction between two different structures on the same set.
 
Let $(X,\sigma,\tau)$ be a solution and $\pi_1$ and $\pi_2$ be two bijections on the set $X$. Let us define, for each $x\in X$, new bijections:
\begin{align}
&\mu_x=\sigma_x\pi_1\quad {\rm and} \quad \nu_x=\tau_x\pi_2.
\end{align} 
Then it is easy to notice that \eqref{birack:1}--\eqref{birack:3} are satisfied for $\mu_x$ and $\nu_x$ if and only if, for $x,y,z\in X$,
\begin{align}
\sigma_x\pi_1\sigma_y&=\sigma_{\sigma_x\pi_1(y)}\pi_1\sigma_{\tau_y\pi_2(x)},\label{gis1}\\
\tau_x\pi_2\tau_y&=\tau_{\tau_x\pi_2(y)}\pi_2\tau_{\sigma_y\pi_1(x)},\label{gis3}\\
\tau_{\sigma_{\tau_y\pi_2(x)}\pi_1(z)}\pi_2\sigma_x\pi_1(y)&=
\sigma_{\tau_{\sigma_y\pi_1(z)}\pi_2(x)}\pi_1\tau_z\pi_2(y).\label{gis2}
\end{align}
\begin{de}
Let $(X,\sigma,\tau)$ be a solution and $\pi_1$ and $\pi_2$ be two bijections on the set $X$ satisfying \eqref{gis1}--\eqref{gis2}. The solution $(X,\mu,\nu)$, with bijections $\mu_x=\sigma_x\pi_1$ and $\nu_x=\tau_x\pi_2$, is called $(\pi_1,\pi_2)$-\emph{isotope} of $(X,\sigma,\tau)$.
In such a case we will say that solutions $(X,\sigma,\tau)$ and $(X,\mu,\nu)$ are \emph{isotopic}. 
\end{de} 

Isotopy is a weaker notion than isomorphism but there are still some properties that remain the same.
For instance, it is evident that isotopes have the same equivalence $\approx$.

\begin{lemma}\label{lm:disiso}
All isotopes of a given solution have the same displacement group.
\end{lemma}
\begin{proof}
Let $(X,\mu,\nu)$ be a $(\pi_1,\pi_2)$-isotope of a solution $(X,\sigma,\tau)$. This means that for each $x\in X$, 
$$\mu_x=\sigma_x\pi_1\quad {\rm and}\quad \nu_x=\tau_x\pi_2.$$
For $x,y\in X$, we have:
$$\mu_x\mu_y^{-1}=\sigma_x\pi_1\pi_1^{-1}\sigma_y^{-1}=\sigma_x\sigma_y^{-1}\quad {\rm and}\quad \nu_x\nu_y^{-1}=\tau_x\pi_2\pi_2^{-1}\tau_y^{-1}=\tau_x\tau_y^{-1}. \qedhere$$
\end{proof}

As expected, being isotopic is a symmetric relation.

\begin{lemma}\label{lem:iso-sym}
	Let a solution $(X,\mu,\nu)$ be a $(\pi_1,\pi_2)$-isotope of a solution $(X,\sigma,\tau)$. Then the solution $(X,\sigma,\tau)$
	is a $(\pi_1^{-1},\pi_2^{-1})$-isotope of $(X,\mu,\nu)$.
\end{lemma}

\begin{proof}
	Evident.
\end{proof}

Analogously we can see that being isotopic is a transitive relation and therefore isotopic solutions form an equivalence class.

\vskip 2mm
We shall focus on our specific classes of solutions only.
In the case of a $2$-permutational solution $(X,\sigma,\tau)$, Condition \eqref{gis2} reduces to:
\begin{align}
&\tau_{\sigma_{a}\pi_1(y)}\pi_2\sigma_x\pi_1=
\sigma_{\tau_{b}\pi_2(x)}\pi_1\tau_y\pi_2,\label{gis31}
\end{align}
for $x,y,a,b\in X$.

	\begin{lemma}\label{lem:2-red_iso}
		Let $(X,\sigma,\tau)$ be a 
		solution and $\pi_1$ and $\pi_2$ be bijections  satisfying Conditions \eqref{gis1}, \eqref{gis3} and \eqref{gis31}. The $(\pi_1,\pi_2)$-isotope $(X,\sigma,\tau)$ is $2$-reductive if and only if, for each $x,y\in X$:
		\begin{align}
			&\sigma_{\sigma_x\pi_1(y)}=\sigma_{\tau_x\pi_2(y)}=\sigma_y,\\
			&\tau_{\sigma_x\pi_1(y)}=\tau_{\tau_x\pi_2(y)}=\tau_y.
		\end{align}
		
	\end{lemma}
	\begin{proof}
		Let us denote by  $(X,L,\mathbf{R})$ the 
		$(\pi_1,\pi_2)$-isotope $(X,\sigma,\tau)$.
		It is $2$-reductive if and only if
		\begin{align*}
			&\sigma_{\sigma_x\pi_1(y)}\pi_1= L_{L_x(y)}=L_y=\sigma_y\pi_1, \\
			& \tau_{\tau_x\pi_2(y)}\pi_2=\mathbf{R}_{\mathbf{R}_x(y)}=\mathbf{R}_y=\tau_{y}\pi_2, \\
			& \sigma_{\tau_x\pi_2(y)}\pi_1=L_{\mathbf{R}_x(y)}=L_y=\sigma_y\pi_1, \\
			& \tau_{\sigma_x\pi_1(y)}\pi_2=\mathbf{R}_{L_x(y)}=\mathbf{R}_y=\tau_y\pi_2,
		\end{align*}
		 for each $x,y\in X$.
	\end{proof}
	
When we are making isotopes of a $2$-reductive solution $(X,L,\mathbf{R})$, the conditions \eqref{gis1}--\eqref{gis2} clearly reduce to the following ones:
\begin{align}
&L_x\pi_1 L_y=L_{\pi_1(y)}\pi_1L_{\pi_2(x)},\label{is1}\\
&\mathbf{R}_x\pi_2 \mathbf{R}_y=\mathbf{R}_{\pi_2(y)}\pi_2\mathbf{R}_{\pi_1(x)},\label{is3}\\
&\mathbf{R}_{\pi_1(y)}\pi_2 L_x\pi_1=L_{\pi_2(x)}\pi_1 \mathbf{R}_y\pi_2,\label{is2}
\end{align}
for $x,y\in X$.

\begin{example}
Let $(X,L,\mathbf{R})$ be a $2$-reductive solution and let $\varphi,\psi\in \langle L_x,\mathbf{R}_y\mid x,y\in X\rangle$. By \cite[Proposition 3.6(iii)]{JP23a}, the group $\langle L_x,\mathbf{R}_y\mid x,y\in X\rangle$ is commutative, hence by $2$-reductivity we obtain
\begin{align*}
&L_x\varphi L_y=L_y\varphi L_x =L_{\varphi(y)}\varphi L_{\psi(x)},\\
&\mathbf{R}_x\psi \mathbf{R}_y=
\mathbf{R}_y\psi \mathbf{R}_x=
\mathbf{R}_{\psi(y)}\psi\mathbf{R}_{\varphi(x)},\\
&\mathbf{R}_{\varphi(y)}\psi L_x\varphi=
L_x\varphi\mathbf{R}_{\varphi(y)}\psi=
L_{\psi(x)}\varphi \mathbf{R}_y\psi,
\end{align*}
for each $x,y\in X$.
Therefore, every pair $(\varphi,\psi)$ of bijections 
$\varphi,\psi\in \langle L_x,\mathbf{R}_y\mid x,y\in X\rangle$,  satisfies Conditions \eqref{is1}--\eqref{is2}.
\end{example}

The previous example was a special case of the following lemma,
since $L_x$ and $\mathbf{R}_x$ are automorphisms of a~2-reductive solution.
	
\begin{lemma}\label{exm:piinv}
Let $(X,L,\mathbf{R})$ be a $2$-reductive solution and let $\pi_1$ and $\pi_2$ be two commuting automorphisms of $(X,L,\mathbf{R})$. 
Then the following conditions are equivalent:
\begin{itemize}
	\item [(i)] the bijections
	$\pi_1$ and $\pi_2$ satisfy Conditions \eqref{is1}--\eqref{is2},
	\item [(ii)] $x\approx \pi_1\pi_2(x)$, for each $x\in X$,
	\item [(iii)] the automorphism $\pi_1\pi_2$ commutes with
	the set $\{L_x,\mathbf{R}_x\mid x\in X\} $,
	\item [(iv)] for any automorphism $\varphi$ of $(X,L,\mathbf{R})$, the automorphism $\varphi^{-1}\pi_1\pi_2\varphi$ commutes with
	the set $\{L_x,\mathbf{R}_x\mid x\in X\} $.
\end{itemize}
\end{lemma}

\begin{proof}
At first note that, for any two commuting automorphisms $\pi_1$ and $\pi_2$ of $(X,L,\mathbf{R})$, Condition \eqref{is2} is always satisfied. Indeed, by commutativity of the group $\langle L_x,\mathbf{R}_y\mid x,y\in X\rangle$, for $x,y\in X$ we have:
\begin{align*}
&\mathbf{R}_{\pi_1(y)}L_{\pi_2 (x)}=L_{\pi_2 (x)}\mathbf{R}_{\pi_1(y)} \quad \Rightarrow\quad \mathbf{R}_{\pi_1(y)}L_{\pi_2 (x)}\pi_2 \pi_1=L_{\pi_2 (x)}\mathbf{R}_{\pi_1(y)}\pi_1 \pi_2 \quad \stackrel{\eqref{isomorphism}}{\Rightarrow}\\
& \mathbf{R}_{\pi_1(y)}\pi_2 L_x\pi_1=L_{\pi_2 (x)}\pi_1\mathbf{R}_{y}\pi_2.
\end{align*}
(i)$\Leftrightarrow$(ii)
  Due to commutativity of $\langle L_x,\mathbf{R}_y\mid x,y\in X\rangle$, we obtain, for $x,y\in X$:
  \begin{multline*}
  	L_x\pi_1 L_y\stackrel{\eqref{is1}}=L_{\pi_1(y)}\pi_1L_{\pi_2(x)} \quad \stackrel{\eqref{isomorphism}}{\Leftrightarrow}\quad
  	L_{\pi_1(y)}L_x\pi_1=L_{\pi_1(y)}L_{\pi_1\pi_2(x)}\pi_1
  	\quad\Leftrightarrow\quad\\
  	L_x=L_{\pi_1\pi_2(x)}
  	\quad\Leftrightarrow\quad
  	x\sim \pi_1\pi_2(x).
  \end{multline*}
The proof that $\eqref{is3}$ is equivalent to the condition $ x\backsim \pi_1\pi_2(x)$, is analogous.

\noindent
(ii)$\Leftrightarrow$(iii) For $x\in X$ we have
$$L_x\stackrel{\mathrm{(ii)}}{=}L_{\pi_1\pi_2(x)}
\stackrel{\eqref{isomorphism}}{=}\pi_1\pi_2 L_x(\pi_1\pi_2)^{-1}
\stackrel{\mathrm{(iii)}}{=}L_x$$
and analogously for $\mathbf{R}_x$.
\\
\noindent
(iv)$\Rightarrow$(iii) is evident. 
\\
\noindent
(iii)$\Rightarrow$(iv) For each $x\in X$,
\[\varphi^{-1}\pi_1\pi_2\varphi L_x
\stackrel{\eqref{isomorphism}}{=}
\varphi^{-1}\pi_1\pi_2 L_{\varphi(x)}\varphi =
\varphi^{-1} L_{\varphi(x)}\pi_1\pi_2\varphi 
\stackrel{\eqref{isomorphism}}{=}
L_x\varphi^{-1}\pi_1\pi_2\varphi
\]
and analogously for $\mathbf{R}_x$.

\end{proof}

In particular, the pair ($\pi$,$\pi^{-1})$ satisfies Conditions \eqref{is1}--\eqref{is2},
for any automorphism $\pi$ of a 2-reductive solution $(X,L,\mathbf{R})$.
\vskip 2mm
Every isotope of a  $2$-reductive solution is a $2$-permutational one.

\begin{lemma}\label{lm:iso2red}
Let $(X,L,\mathbf{R})$ be a $2$-reductive solution and $\pi_1$ and $\pi_2$ be two bijections on the set $X$ satisfying \eqref{is1}--\eqref{is2}. Then the $(\pi_1,\pi_2)$-isotope of $(X,L,\mathbf{R})$ is a $2$-permutational solution.
\end{lemma}
\begin{proof}
Let $(X,\sigma,\tau)$ be $(\pi_1,\pi_2)$-isotope of $(X,L,\mathbf{R})$. For $x,y,z\in X$
\begin{align*}
&\sigma_{\sigma_x(z)}=  L_{L_x\pi_1(z)}\pi_1=L_{\pi_1(z)}\pi_1=L_{L_y\pi_1(z)}\pi_1=\sigma_{\sigma_y(z)} \quad {\rm and}\\
&\sigma_{\tau_x(z)}=  L_{\mathbf{R}_x\pi_2(z)}\pi_1=L_{\pi_2(z)}\pi_1=L_{\mathbf{R}_y\pi_2(z)}\pi_1=\sigma_{\tau_y(z)}.
\end{align*}
Similarly, we can check that also 
$\tau_{\tau_x(z)}=\tau_{\tau_y(z)}$ and $ \tau_{\sigma_x(z)}=\tau_{\sigma_y(z)}$, which shows that $(X,\sigma,\tau)$ is $2$-permutational.
\end{proof}


The previous result can be reverted, in the sense that every $2$-permutational solution admits a $2$-reductive isotope.

\begin{proposition}\label{prop:2p_iso1}
Let $(X,\sigma,\tau)$ be a $2$-permutational solution and let $e\in X$. 
Then, the pair $(\sigma_e^{-1},\tau_e^{-1})$
satisfies Conditions \eqref{gis1}, \eqref{gis3} and \eqref{gis31}
and the $(\sigma_e^{-1},\tau_e^{-1})$-isotope is $2$-reductive.
\end{proposition}
\begin{proof}
Let $x,y,z,a,b\in X$. 
\\
(1) At first note that, for the $2$-permutational solution $(X,\sigma,\tau)$, we have: 
\begin{align*}
&\sigma_x\sigma_e^{-1}\sigma_y\stackrel{\eqref{eq:com1}}=\sigma_y\sigma_e^{-1}\sigma_x=
\sigma_{\sigma_e\sigma_e^{-1}(y)}\sigma_e^{-1}\sigma_{\tau_e\tau_e^{-1}(x)}\stackrel{\eqref{eq:per1},\eqref{eq:per3}}=
\sigma_{\sigma_x\sigma_e^{-1}(y)}\sigma_e^{-1}\sigma_{\tau_y\tau_e^{-1}(x)},
\end{align*}
and similarly $\tau_x\tau_e^{-1}\tau_y=\tau_{\tau_x\tau_e^{-1}(y)}\tau_e^{-1}\tau_{\sigma_y\sigma_e^{-1}(x)}$. Moreover,
\begin{multline*}
\tau_{\sigma_a\sigma_e^{-1}(z)}\tau_e^{-1}\sigma_x\sigma_e^{-1}
\stackrel{\eqref{eq:per1}}=
\tau_{\sigma_e\sigma_e^{-1}(z)}\tau_e^{-1}\sigma_x\sigma_e^{-1}=
\tau_{z}\tau_e^{-1}\sigma_x\sigma_e^{-1}
\stackrel{\eqref{eq:com2}}=\\
\sigma_x\sigma_e^{-1}\tau_{z}\tau_e^{-1}=
\sigma_{\tau_e\tau_e^{-1}(x)}\sigma_e^{-1}\tau_{z}\tau_e^{-1}
\stackrel{\eqref{eq:per2}}=
\sigma_{\tau_b\tau_e^{-1}(x)}\sigma_e^{-1}\tau_{z}\tau_e^{-1}
\end{multline*}
which shows that the pair $(\sigma_e^{-1},\tau_e^{-1})$ satisfies Conditions \eqref{gis1}, \eqref{gis3} and \eqref{gis31}. 
	Now we use Lemma~\ref{lem:2-red_iso}:
\begin{align*}
	&\sigma_{\sigma_x\sigma_e^{-1}(y)}\stackrel{\eqref{eq:per1}}=\sigma_{\sigma_e\sigma_e^{-1}(y)}=\sigma_y=\sigma_{\tau_e\tau_e^{-1}(y)}\stackrel{\eqref{eq:per3}}=
	\sigma_{\tau_x\tau_e^{-1}(y)},\\
	&\tau_{\sigma_x\sigma_e^{-1}(y)}
	\stackrel{\eqref{eq:per4}}=\tau_{\sigma_e\sigma_e^{-1}(y)}=\tau_y=\tau_{\tau_e\tau_e^{-1}(y)}\stackrel{\eqref{eq:per2}}=
	\tau_{\tau_x\tau_e^{-1}(y)},
\end{align*}
and the isotope is $2$-reductive.
\end{proof}

Every $2$-permutational solution is isotopic to a $2$-reductive
one and every $2$-reductive solution can be constructed using a
combinatorial condstruction described in~\cite{JP23a}.
This gives us a theoretic means of constructing all $2$-permutational solutions of a given size: we construct all the $2$-reductive solutions
and then we construct all their isotopes. However, an isotopy class often contains several non-isomorphic $2$-reductive solutions
as we have already observed in~\cite{JPZ20a}
. To obtain an effective isomorphism criterion, we have to look at different isotopes than 
those described in Proposition~\ref{prop:2p_iso1}.

%

\section{Isotopes of square free $2$-reductive solutions}\label{sec:isosqfree}

A main drawback of the previous section is: given two isotopes of two $2$-reductive solutions, we lack an effective criterion to determine whether they are isomorphic. The same problem occured already in~\cite{JPZ20a} and it was responded by Rump in~\cite{Rump22},
	where he proved that every involutive solution of multipermutation level~2
	is isotopic to an involutive square-free 2-reductive solution. Rump's
	approach is different than the ours, he uses so-called cycle sets. Nevertheless, when
	one analyzes the result, it turns out that the permutation used
	to construct the isotope is actually the diagonal permutation.

\begin{de}[\cite{ESS}]
	Let $(X,\sigma,\tau)$ be a solution. We define
	\begin{equation}
		U\colon X\to X,\ U(x)= \sigma^{-1}_x(x) \quad\text{ and }\quad T\colon X\to X,\ T(x)= \tau^{-1}_x(x).
	\end{equation}
\end{de}

	 Etingof,  Schedler and Soloviev observed in \cite[Proposition 2,2]{ESS} that, for involutive solutions, the mappings  are bijections on $X$. The same is also true for non-involutive solutions (see \cite{JP24}),
	 namely we have 
	 \begin{equation}
	 	U^{-1}(x)= \sigma_{\tau_x^{-1}(x)}(x)
	 	\quad\text{ and }\quad T^{-1}(x)= \tau_{\sigma_x^{-1}(x)}(x)
	 	.
	 \end{equation}
	The proof in~\cite{JP24} is quite technical and therefore we present a shorter proof for the $2$-permutational case here.
	Indeed:
	\begin{align*}
		&U^{-1}U(x)=U^{-1}(\sigma^{-1}_x(x))=\sigma_{\tau^{-1}_{\sigma^{-1}_x(x)}(\sigma^{-1}_x(x))}\sigma^{-1}_x(x)\stackrel{\eqref{eq:per1}}=
		\sigma_{\tau_{x}^{-1}(\sigma^{-1}_x(x))}\sigma^{-1}_x(x)\stackrel{\eqref{eq:tauinverse1}}=\\
		&\sigma_{\sigma_{x}(\sigma^{-1}_x(x))}\sigma^{-1}_x(x)=\sigma_x\sigma^{-1}_x(x)=x,
	\end{align*}
	and 
	\begin{align*}
		UU^{-1}(x)
		=U(\sigma_{\tau_x^{-1}(x)}(x))
		\stackrel{\eqref{eq:tauinverse1}}
		=U(\sigma_{\sigma_x(x)}(x))=\sigma^{-1}_{\sigma_{\sigma_x(x)}(x)}\sigma_{\sigma_x(x)}(x)\stackrel{\eqref{eq:per1}}=\sigma^{-1}_{\sigma_x(x)}\sigma_{\sigma_x(x)}(x)=x,
	\end{align*}
	and similarly for $T$.
	Moreover, the mappings $U$ and $T$ commute. Indeed, for $x\in X$,
	\begin{multline*}
		UT(x)=U\tau_x^{-1}(x)=\sigma^{-1}_{\tau^{-1}_x(x)}\tau_x^{-1}(x)\stackrel{\eqref{eq:tauinverse4}}=
		\tau^{-1}_{\sigma_x^{-1}(x)}\sigma^{-1}_{\sigma^{-1}_x\tau^{-1}_x(x)}(x)\stackrel{\eqref{eq:tauinverse3}}=\\
		\tau^{-1}_{\sigma_x^{-1}(x)}\sigma^{-1}_{x}(x)=T\sigma_x^{-1}(x)=TU(x).
	\end{multline*}

	In the case of $2$-permutational solutions, the diagonal permutations are actually automorphisms.
	
	\begin{lemma}\label{lem:UT_auto1}
		The permutations $U$ and~$T$ are automorphisms of  a $2$-permutational solution $(X,\sigma,\tau)$.
	\end{lemma}	
	\begin{proof}
		We check~\eqref{isomorphism}. For $x,y\in X$:
		\begin{align*}
			U\sigma_x(y)&=\sigma^{-1}_{\sigma_x(y)}\sigma_x(y)
			\stackrel{\eqref{birack:1}}{=}
			\sigma_{\tau_y(x)}\sigma_y^{-1}(y)
			\stackrel{\eqref{eq:sigmasigmainverse}}{=}
			\sigma_{\sigma^{-1}_x(x)}\sigma_y^{-1}(y)
			=\sigma_{U(x)}U(y),\\
			U\tau_x(y)&=\sigma^{-1}_{\tau_x(y)}\tau_x(y)
			\stackrel{\eqref{eq:41}}=
			\tau_{\sigma_x^{-1}(x)}\sigma^{-1}_{\sigma_x\tau_x(y)}(y)
			\stackrel{\eqref{eq:tauinverse3}}{=}
			\tau_{\sigma_x^{-1}(x)}\sigma^{-1}_{y}(y)
			=\tau_{U(x)}U(y).
			\end{align*}
		The proof for the permutation~$T$ is analogous.
	\end{proof}

\begin{example}\label{exm:irretractable}
			In general, the permutations $U$ and $T$ are not automorphisms.
			Let $(X,\sigma,\tau)$ be the following irretractable solution on $\{0,1,2,3,4\}$:
			\begin{align*}
				\sigma_0&=\tau_1=\tau_4=\mathrm{id}, &
				\sigma_1&=\sigma_3=(0,2,3),\\
				\tau_0&=\tau_2=\tau_3=(1,4)(2,3), &
				\sigma_2&=\sigma_4=(0,3,2).
			\end{align*}
			Then $U=(2,3)$ and $U\sigma_1(0)=3$ whereas $\sigma_{U(1)}U(0)=2$ and therefore $U$ is not an automorphism of this solution.
		\end{example}

	Rump's result can be straightforwardly generalized.
		
\begin{proposition}\label{prop:2p_iso2}
	Let $(X,\sigma,\tau)$ be a $2$-permutational solution
	Then, the pair $(U,T)$
		satisfies Conditions \eqref{gis1}, \eqref{gis3} and \eqref{gis31}
		and the $(U,T)$-isotope of $(X,\sigma,\tau)$ is square-free $2$-reductive.
\end{proposition}

\begin{proof}
Let $x,y,z\in X$. Then, for Condition~\eqref{gis1}, we have:
\begin{align*}
	&\sigma_x U\sigma_y(z)=\sigma_x\sigma^{-1}_{\sigma_y(z)}\sigma_y(z)\stackrel{\eqref{eq:com1}}=\sigma_y\sigma^{-1}_{\sigma_y(z)}\sigma_x(z)
	\stackrel{\eqref{eq:per1}}=\\
	&\sigma_{\sigma_y\sigma^{-1}_y(y)}\sigma^{-1}_{\sigma_{x}(z)}\sigma_{\tau_x\tau^{-1}_x(x)}(z)\stackrel{\eqref{eq:per1}+\eqref{eq:per3}}=\\
	&\sigma_{\sigma_x\sigma^{-1}_y(y)}\sigma^{-1}_{\sigma_{\tau_y\tau^{-1}_x(x)}(z)}\sigma_{\tau_y\tau^{-1}_x(x)}(z)=\sigma_{\sigma_xU(y)}U \sigma_{\tau_yT(x)}(z).
\end{align*}
Similarly we can prove that $\tau_xT \tau_y=\tau_{\tau_xT(y)}T\tau_{\sigma_y U(x)}$ is satisfied too. For~\eqref{gis31}:
\begin{align*}
	&\tau_{\sigma_{a}U(z)}T\sigma_xU(y)=\\
	&\tau_{\sigma_{a}\sigma_z^{-1}(z)}\tau^{-1}_{\sigma_x\sigma^{-1}_y(y)}\sigma_x\sigma_y^{-1}(y)\stackrel{\eqref{eq:per3}+\eqref{eq:per4}}=
	\tau_{\sigma_{z}\sigma_z^{-1}(z)}\tau^{-1}_{\sigma_y\sigma^{-1}_y(y)}\sigma_x\sigma_y^{-1}(y)=\\
	&
	\tau_{z}\tau^{-1}_{y}\sigma_x\sigma_y^{-1}(y)\stackrel{\eqref{eq:com2}}=\sigma_x\sigma_y^{-1}\tau_{z}\tau^{-1}_{y}(y)=\sigma_{\tau_{x}\tau_x^{-1}(x)}\sigma^{-1}_{\tau_y\tau^{-1}_y(y)}\tau_z\tau_y^{-1}(y)
	\stackrel{\eqref{eq:per3}+\eqref{eq:per4}}=
	\\
	& 
	\sigma_{\tau_{b}\tau^{-1}_x(x)}\sigma^{-1}_{\tau_z\tau^{-1}_y(y)}\tau_z\tau_y^{-1}(y)= 
	\sigma_{\tau_{b}T(x)}U\tau_zT(y),
\end{align*}
and hence Conditions \eqref{gis1}, \eqref{gis3} and \eqref{gis31} are satisfied. 
	Now, for $x,y\in X$,
\begin{align*}		
	&\sigma_{\sigma_xU(y)}=\sigma_{\sigma_x\sigma_y^{-1}(y)}
	\stackrel{\eqref{eq:per1}}=\sigma_{\sigma_y\sigma_y^{-1}(y)}=
	\sigma_y=\sigma_{\tau_y\tau_y^{-1}(y)}\stackrel{\eqref{eq:per3}} =\sigma_{\tau_x\tau_y^{-1}(y)}=\sigma_{\tau_xT(y)}\\
	& \tau_{\tau_xT(y)}=\tau_{\tau_x\tau_y^{-1}(y)}\stackrel{\eqref{eq:per2}}=\tau_{\tau_y\tau_y^{-1}(y)}=\tau_y= \tau_{\sigma_y\sigma_y^{-1}(y)}\stackrel{\eqref{eq:per4}}=\tau_{\sigma_x\sigma_y^{-1}(y)}=\tau_{\sigma_xU(y)},
\end{align*}
and the $(U,T)$-isotope is $2$-reductive, according to Lemma~\ref{lem:2-red_iso}. Finally, 
\[
	\sigma_x U(x)=\sigma_x\sigma_x^{-1}(x)=x\qquad\text{and}\qquad
	\tau_x T(x)=\tau_x\tau^{-1}_x(x)=x,
\]
and the isotope is square-free.
\end{proof}

\begin{proposition}\label{prop:unsfree}
Let $(X,\sigma,\tau)$ be a $2$-permutational solution. 
Then there exists exactly one pair of permutations $(\pi_1,\pi_2)$ such that
the $(\pi_1,\pi_2)$-isotope of $(X,\sigma,\tau)$ is square-free.
\end{proposition}

\begin{proof}
The existence was proved in~Proposition~\ref{prop:2p_iso2}.
Let $(X,L,\mathbf{R})$ be a square free $(\pi_1,\pi_2)$-isotope of $(X,\sigma,\tau)$, for some bijections $\pi_1$ and $\pi_2$
Then, for each $x\in X$ 
\begin{align*}
&x=L_x(x)=\sigma_x\pi_1(x)\quad {\rm and}\quad x=\mathbf{R}_x(x)=\tau_x\pi_2(x).
\end{align*}
Hence 
\begin{align*}
&\pi_1(x)=\sigma_x^{-1}(x)=U(x)\quad {\rm and}\quad \pi_2(x)=\tau_x^{-1}(x)=T(x).&\qedhere
\end{align*}
\end{proof}




\begin{lemma}\label{lem:UT_auto2}
Let	$(X,\sigma,\tau)$ be a $2$-permutational solution and let $U$ and $T$ be its diagonals. Moreover, let $(X,L,\mathbf{R})$ be the square-free isotope of $(X,\sigma,\tau)$. Then $U$ and $T$ are automorphisms of~$(X,L,\mathbf{R})$.
\end{lemma}
\begin{proof}
By commutativity of $U$ and $T$ and 
Lemma~\ref{lem:UT_auto1}, for each~$x\in X$, we have
	\[
	L_{T(x)}T=\sigma_{T(x)}UT=\sigma_{T(x)}TU\stackrel{\eqref{isomorphism}}{=}
	T\sigma_x U=TL_x,
	\]
	and analogously
	for the mapping~$\mathbf{R}_x$. Hence $T$ is an automorphism of the solution $(X,L,\mathbf{R})$. The proof for~$U$ is symmetric.
\end{proof}

We have now a clearer algorithm how to construct all the solutions of multipermutation level~$2$. Each such a solution~$(X,\sigma,\tau)$ is determined by the triple $((X,L,\mathbf{R}),U,T)$ where $U$ and $T$ are its diagonal mappings and $(X,L,\mathbf{R})$ is the unique square-free solution
	isotopic to $(X,\sigma,\tau)$. And it turns out that two such triples yield isomorphic solutions if the triples are conjugated by an isomorphism.

\begin{proposition}\label{thm:iso}
Let $(X,L,\mathbf{R})$ and $(X',L',\mathbf{R}')$ be two $2$-reductive square-free solutions, let $\pi_1,\pi_2\in\Aut((X,L,\mathbf{R}))$
and let $\pi_1',\pi_2'\in\Aut((X',L',\mathbf{R}'))$, both the pairs satisfying \eqref{is1}--\eqref{is2}. 
Then the $(\pi_1,\pi_2)$-isotope of $(X,L,\mathbf{R})$ and the $(\pi_1^{'}\pi_2^{'})$-isotope of $(X',L',\mathbf{R}')$ are isomorphic $2$-permutational solutions 
if and only if there is an isomorphism $\varphi$ from $(X,L,\mathbf{R})$ onto $(X',L',\mathbf{R}')$
such that
\begin{enumerate}
\item $\pi_1^{'}=\varphi\pi_1\varphi^{-1}$
\item $\pi_2^{'}=\varphi\pi_2\varphi^{-1}$
\end{enumerate}
\end{proposition}
\begin{proof}
Let $(X,\sigma,\tau)$ be $(\pi_1,\pi_2)$-isotope of $(X,L,\mathbf{R})$, $(X',\sigma^{'},\tau^{'})$ be $(\pi_1^{'}\pi_2^{'})$-isotope of $(X',L',\mathbf{R}')$ and $\varphi\colon X\to X'$ be an isomorphism of $(X,\sigma,\tau)$ and $(X',\sigma^{'},\tau^{'})$. 
We prove that~$\varphi$ is an isomorphism from $(X,L,\mathbf{R})$ onto $(X',L',\mathbf{R}')$ satisfying (1) and (2).
For $x,y\in X$ and $z\in X'$ we have:
\begin{align*}
&\varphi L_x\pi_1(y)=\varphi \sigma_x(y) \stackrel{\eqref{isomorphism}}=\sigma_{\varphi(x)}^{'}\varphi(y)=L'_{\varphi(x)}\pi_1^{'}\varphi(y) \quad \stackrel{x\mapsto \pi_1(y)}\Rightarrow\\
&\varphi\pi_1(y)=\varphi\pi_1L_y(y)\stackrel{\eqref{isomorphism}}=\varphi L_{\pi_1(y)}\pi_1(y)
=L'_{\varphi(\pi_1(y))}\pi_1^{'}\varphi(y) \quad \stackrel{y\mapsto \pi^{-1}_1\varphi^{-1}(z)}\Rightarrow\\
&z=\varphi\pi_1\pi^{-1}_1\varphi^{-1}(z)=L'_z\pi_1^{'}\varphi \pi_1^{-1}\varphi^{-1}(z)\quad \Rightarrow\\
&z={L'_z}^{-1}(z)=\pi_1^{'}\varphi \pi_1^{-1}\varphi^{-1}(z)\quad \Rightarrow\quad \pi_1^{'}=\varphi \pi_1\varphi^{-1}.
\end{align*}
Moreover, according to the first and the last lines of the previous computation,
\begin{align*}
L'_{\varphi(x)}=\varphi L_x\pi_1\varphi^{-1}(\pi_1^{'})^{-1}= \varphi L_x\pi_1\varphi^{-1}\varphi\pi_1^{-1}\varphi^{-1}=\varphi L_x\varphi^{-1}.
\end{align*}
Similarly, we obtain that $\mathbf{R}'_{\varphi(x)}=\varphi \mathbf{R}_x\varphi^{-1}$ and $\pi_2^{'}=\varphi\pi_2\varphi^{-1}$. 

On the other hand, if there is an automorphism $\varphi\colon X\to X$ of the solution $(X,L,\mathbf{R})$ such that the conditions $(1)$ and $(2)$ are satisfied then 
\begin{align*}
\sigma_{\varphi(x)}^{'}\varphi=L'_{\varphi(x)}\pi_1^{'}\varphi=\varphi L_x\varphi^{-1}\varphi \pi_1\varphi^{-1}\varphi=\varphi L_x \pi_1=\varphi \sigma_x,
\end{align*}
and similarly, $\tau_{\varphi(x)}^{'}\varphi=\varphi \tau_x$, which completes the proof.
\end{proof}

Combining Proposition~\ref{prop:2p_iso2}, Proposition~\ref{prop:unsfree}, Lemma~\ref{lem:iso-sym}, Lemma~\ref{lem:UT_auto2} and Proposition~\ref{thm:iso} we obtain:
\begin{theorem}
	\label{th:repres}
	Let $(X,\sigma,\tau)$ be a $2$-permutational solution. Then there exist a unique, up to isomorphism, square-free $2$-reductive solution $(X,L,\mathbf{R})$ and a unique, up to conjugacy,
	pair of commuting automorphisms $(\pi_1,\pi_2)\in\Aut((X,L,\mathbf{R}))^2$ such that
	$(X,\sigma,\tau)$ is the $(\pi_1,\pi_2)$-isotope of $(X,L,\mathbf{R})$.
\end{theorem}

\begin{rem}
	It is not true in general that every solution admits a square-free isotope. Take the solution $(X,\sigma,\tau)$ defined in~Example~\ref{exm:irretractable}. If we try to construct the $(U,T)$-isotope $(X,L,\mathbf{R})$ then we do not get a solution since
	\begin{align*}
		L_0L_1&=\sigma_{0} U\sigma_{1} U=(2,3) (0,2,3) (2,3)=(0,3,2),\\
		L_{L_0(1)}L_{\mathbf{R}_1(0)}&=\sigma_{\sigma_0U(1)}U\sigma_{\tau_1T(0)}U=\sigma_{1}U\sigma_{0}U=(0,2,3).
	\end{align*}
\end{rem}

\section{$2$-permutational involutive solutions}\label{sec:invol}

In this section we look at the results from the previous section and present conditions under which we obtain an involutive solution.
Recall that a solution $(X,\sigma,\tau)$ is involutive if, for each $x,y\in X$, 
$\tau_y(x)=\sigma_{\sigma_x(y)}^{-1}(x)$  or equivalently, $\sigma_x(y)=\tau^{-1}_{\tau_y(x)}(y)$. In the case of $2$-permutational solutions, it means that for each $x,y\in X$, 
\begin{align}
\tau_y=\sigma_{\sigma_x(y)}^{-1}\quad {\rm or \; equivalently} \quad\sigma_x=\tau^{-1}_{\tau_y(x)}.\label{eq:invper}
\end{align}

\begin{lemma}\cite[Lemma 3.7]{JP23a}\label{lm:lir}
A $2$-reductive solution is involutive if and only if it satisfies {\bf lri}.
\end{lemma}

\begin{lemma}
Let $(X,\sigma,\tau)$ be a $2$-permutational involutive solution. Then, for $x,y\in X$,
\begin{align}
&\sigma_x\tau_y\tau_x=\tau_{\sigma_x(y)}\label{eq:stt}\\
&\tau_y\tau^{-1}_x=\sigma_x\tau_{\sigma^{-1}_x(y)}.\label{eq:tt}
\end{align}
\end{lemma}
\begin{proof}
For $x,y,z\in X$ we have
\begin{align*}
&\eqref{eq:stt}:&&\sigma_x\tau_y\tau_x\stackrel{\eqref{eq:tauinverse4}}=\tau_{\sigma_x(y)}\sigma_{\sigma_z(x)}\tau_x\stackrel{\eqref{eq:invper}}
=\tau_{\sigma_x(y)}.\\
&\eqref{eq:tt}: &&\sigma_x\tau_y\stackrel{\eqref{eq:stt}}=\tau_{\sigma_x(y)}\tau^{-1}_x\quad
\stackrel{y\mapsto \sigma^{-1}_x(y)}\Rightarrow\quad \sigma_x\tau_{\sigma^{-1}_x(y)}=\tau_{y}\tau^{-1}_x. \qedhere
\end{align*}
\end{proof}

\begin{proposition}
A $2$-reductive $(\pi_1,\pi_2)$-isotope of a $2$-permutational solution $(X,\sigma,\tau)$ is involutive if and only if, for each $x\in X$,
\begin{align}
&\pi_2^{-1}=\sigma_x\pi_1\tau_x.\label{eq:forpi2}
\end{align}

\end{proposition}
\begin{proof}
Let $(X,L,\mathbf{R})$ be $2$-reductive $(\pi_1,\pi_2)$-isotope of $(X,\sigma,\tau)$.
So, for each $x\in X$,
\begin{align*}
&L_x=\sigma_x\pi_1\quad {\rm and}\quad \mathbf{R}_x=\tau_x\pi_2.
\end{align*}
By Lemma \ref{lm:lir}, $(X,L,\mathbf{R})$ is involutive if and only if it satisfies {\bf lri}:
\begin{align*}
\mathbf{R}_x=L_x^{-1}\quad \Leftrightarrow\quad \pi_1^{-1}\sigma_x^{-1}=L_x^{-1}=\mathbf{R}_x=\tau_x\pi_2\quad \Leftrightarrow\quad \tau_x^{-1}\pi_1^{-1}\sigma_x^{-1}=\pi_2.
\end{align*}
\end{proof}

\begin{corollary}
Let $(X,\sigma,\tau)$ be a $2$-permutational involutive solution and let $e\in X$. Then $2$-reductive $(\sigma^{-1}_e,\tau^{-1}_e)$-isotope of $(X,\sigma,\tau)$ is also involutive. 
\end{corollary}

\begin{proof}
Let $(X,L,\mathbf{R})$ be the $(\sigma^{-1}_e,\tau^{-1}_e)$-isotope of $(X,\sigma,\tau)$. Hence, for $x\in X$

\begin{align*}
&L_x\mathbf{R}_x=\sigma_x\sigma_e^{-1}\tau_x\tau^{-1}_e\stackrel{\eqref{eq:tt}}=\sigma_x\sigma^{-1}_e\sigma_e\tau_{\sigma^{-1}_e(x)}\stackrel{\eqref{eq:per4}}=\sigma_x\tau_{\sigma^{-1}_y(x)}\stackrel{\eqref{eq:invper}}=\sigma_x\sigma^{-1}_x=\id. \qedhere
\end{align*}
\end{proof}

We have even more. If $(X,\sigma,\tau)$ is a $2$-reductive involutive $(\sigma^{-1}_e,\pi_2)$-isotope of a $2$-permutational solution $(X,\sigma,\tau)$ then by \eqref{eq:forpi2}, $\pi_2^{-1}=\sigma_x\sigma_e^{-1}\tau_x$, for each $x\in X$. Hence $x=e$ forces 
\begin{align*}
\pi_2=(\sigma_e\sigma_e^{-1}\tau_e)^{-1}=\tau_e^{-1}.
\end{align*}

\begin{corollary}\label{cor:UT}
Let $(X,\sigma,\tau)$ be a $2$-permutational involutive solution. Then the $2$-reductive $(U,T)$-isotope of $(X,\sigma,\tau)$ is also involutive. 
\end{corollary}

\begin{proof}
Let $(X,L,\mathbf{R})$ be $(U,T)$-isotope of $(X,\sigma,\tau)$. Hence for $x,y\in X$
\begin{multline*}
L_x\mathbf{R}_x=\sigma_xU\tau_xT(y)=
\sigma_x\sigma^{-1}_{\tau_x\tau_y^{-1}(y)}\tau_x\tau_y^{-1}(y)
\stackrel{\eqref{eq:per3}}=\\
\sigma_x\sigma^{-1}_{y}\tau_x\tau_y^{-1}(y)
\stackrel{\eqref{eq:tt}}=
\sigma_x\tau_{\sigma^{-1}_y(x)}(y)\stackrel{\eqref{eq:invper}}=\sigma_x\sigma^{-1}_x(y)=y. \qedhere
\end{multline*}
\end{proof}

\begin{proposition}\label{thm:invred}
Let $(X,\sigma,\tau)$ be a $(\pi_1,\pi_2)$-isotope of a $2$-reductive  involutive solution $(X,L,\mathbf{R})$. Then $(X,\sigma,\tau)$ is involutive if and only if, for each $x\in X$,
\begin{align}\label{eq:Lpi2}
\pi_2=L_x\pi_1^{-1}L^{-1}_{\pi_1(x)}.
\end{align}
\end{proposition}

\begin{proof}
Let $(X,\sigma,\tau)$ be a $(\pi_1,\pi_2)$-isotope of a $2$-reductive involutive solution $(X,L,\mathbf{R})$. Then, for each $x\in X$, $\sigma_x=L_x\pi_1$, $\tau_x=\mathbf{R}_x\pi_2$ and $L_x=\mathbf{R}^{-1}_x$. This implies 
$\tau_x=\mathbf{R}_x\pi_2=L_x^{-1}\pi_2$. Therefore, for $y\in X$, we have: \begin{align*}
&\tau_y=\sigma^{-1}_{\sigma_x(y)}\quad \Leftrightarrow\quad L_y^{-1}\pi_2=\pi_1^{-1}L^{-1}_{L_x\pi_1(y)}
\quad \Leftrightarrow\quad L_y^{-1}\pi_2=\pi_1^{-1}L^{-1}_{\pi_1(y)}.
\end{align*}
\end{proof}

By results of Section~\ref{sec:isosqfree} and Corollary \ref{cor:UT} we have seen that every involutive $2$-permutational solution can be obtained as some $(\pi,\pi^{-1})$-isotope of a $2$-reductive solution $(X,L,\mathbf{R})$, where $\pi$ is an automorphism of $(X,L,\mathbf{R})$. This necessary condition turns out to be sufficient.

\begin{corollary}\label{cor:invol-ex}
Let $(X,\sigma,\tau)$ be a $(\pi_1,\pi_2)$-isotope of a $2$-reductive square free involutive solution $(X,L,\mathbf{R})$. The solution $(X,\sigma,\tau)$ is involutive if and only if
$\pi_2=\pi_1^{-1}$ and $\pi_1$ is an automorphism of $(X,L,\mathbf{R})$.
\end{corollary}
\begin{proof}
Suppose that $(X,\sigma,\tau)$ is involutive.
By Proposition \ref{thm:invred}, for each  $y\in X$: \begin{align*}
&\pi_2\stackrel{\eqref{eq:Lpi2}}=L_y\pi_1^{-1}L^{-1}_{\pi_1(y)}\quad \Rightarrow\quad \pi_2\pi_1(y)=L_y\pi_1^{-1}L^{-1}_{\pi_1(y)}\pi_1(y)=
L_y\pi_1^{-1}\pi_1(y)=L_y(y)=y.
\end{align*}
Hence $\pi_1=\pi_2^{-1}$ and Proposition 6.6
requires $\pi_1L_x=L_{\pi_1(x)}\pi_1$ which is equivalent to $\pi_1$ being an automorphism of $(X,L,\mathbf{R})$.
The other direction is straightforward.
\end{proof}


Directly by Theorem \ref{thm:iso} we obtain the following.
\begin{corollary}\label{cor:invol-uni}
Let $(X,L,\mathbf{R})$ be a $2$-reductive square free involutive solution and $\pi,\psi$ be bijections on $X$ satisfying \eqref{is1}--\eqref{is2}.
Then $(\pi,\pi^{-1})$-isotope and $(\psi,\psi^{-1})$-isotope of $(X,L,\mathbf{R})$ are isomorphic $2$-permutational involutive solutions if and only if there is an automorphism $\varphi\colon X\to X$ of the solution $(X,L,\mathbf{R})$ such that
\begin{align*}
&\psi=\varphi\pi\varphi^{-1}.
\end{align*}
\end{corollary}

\section{Algorithm}\label{sec:algorithm}



In this section we summarize the algorithms how to obtain all (involutive or non-involutive) solutions of multipermutation level~2. Since solutions of multipermutation level~1 are exactly those that are isotopic to a trivial solution, we can omit trivial solutions from our algorithms, obtaining only solutions of multipermutation level exactly~2.

Let us start from the involutive case. 
By results from \cite{JPZ20a} all involutive square free $2$-reductive solutions of a~given size~$n$, up to isomorphism, can be obtained using
the following algorithm:
\begin{enumerate}
 \item For all partitionings $n=n_1+n_2+\cdots +n_k$ do (2)--(5).
 \item For all abelian groups $A_1$, \dots ,$A_k$ of size $|A_i|=n_i$ do (3)--(5).
 \item For all constants $c_{i,j}\in A_j$ with $c_{i,i}=0$, $1\leq i,j\leq k$,  do (4).
 \item If, for all $1\leq j\leq k$, we have $A_j=\langle \{c_{i,j}\mid 1\leq i\leq k\}\rangle$
 then construct a solution $(\bigcup A_i,\sigma,\tau)$ according to: $\sigma_x(y)=y+c_{i,j}\quad {\rm and} \quad \tau_y(x)= x-c_{j,i}$ for $x\in A_i$ and $y\in A_j$.
 \item For any permutation~$\alpha$ of $\{1,\ldots,k\}$
 	and any isomorphisms $\psi_i:A_i\to A_{\alpha(i)}$, the constants 
 	$\{\psi_j(c_{\alpha(i),\alpha(j)})\}$ yield an isomorphic solution and therefore shall not be considered in the sequel.
\end{enumerate}

An algorithm for constructing all involutive solutions of multipermutation level~2, up to isomorphism, according to Theorem~\ref{th:repres}, Corollary~\ref{cor:invol-ex} and Corollary~\ref{cor:invol-uni} is the following:
\begin{enumerate}
	\item For each non-trivial square-free 2-reductive involutive solution $(X,L,\mathbf{R})$ do (2)--(3).
	\item Compute the conjugacy classes of $\Aut((X,L,\mathbf{R}))$.
	\item For each conjugacy class choose a~representative $\pi$ and
	return the $(\pi,\pi^{-1})$-isotope of~$(X,L,\mathbf{R})$.
\end{enumerate}

By results from \cite{JP23a} we can construct all, not necessarily involutive, square free $2$-reductive solutions of a~given size~$n$, up to isomorphism, using
the following algorithm:
\begin{enumerate}
 \item For all partitionings $n=n_1+n_2+\cdots +n_k$ do (2)--(5).
 \item For all abelian groups $A_1$, \dots ,$A_k$ of size $|A_i|=n_i$ do (3)--(5).
 \item For all constants $c_{i,j},d_{i,j}\in A_j$ with $c_{i,i}=d_{i,i}=0$, $1\leq i,j\leq k$,  do (4).
 \item If, for all $1\leq j\leq k$, we have $A_j=\langle \{c_{i,j},d_{i,j}\mid 1\leq i\leq k\}\rangle$
 then construct a solution $(\bigcup A_i,\sigma,\tau)$ according to: $\sigma_x(y)=y+c_{i,j}\quad {\rm and} \quad \tau_y(x)= x+d_{j,i}$ for $x\in A_i$ and $y\in A_j$.
 \item For any permutation~$\alpha$ of $\{1,\ldots,k\}$
 	and any isomorphisms $\psi_i:A_i\to A_{\alpha(i)}$, the pair of constants
 $\{(\psi_j(c_{\alpha(i),\alpha(j)}),\psi_i(d_{\alpha(i),\alpha(j)}))\}$ yields an isomorphic solution and therefore shall not be considered in the sequel.
\end{enumerate}

For non-involutive solutions of multipermutation level~2, the algorithm is more complex 
since we have to find a pair of commuting automorphisms $(\pi_1,\pi_2)$, up to conjugacy by an automorphism. It is easy to see that we can choose $\pi_1$, up to conjugacy by an automorphism and then we choose $\pi_2$ up to conjugacy by an automorphism that commutes with $\pi_1$. A key role is played by Lemma~\ref{exm:piinv}.
	\begin{enumerate}
		\item For each non-trivial square-free 2-reductive solution $(X,L,\mathbf{R})$ do (2)--(7).
		\item Construct the set $\{L_x,\mathbf{R}_x\mid x\in X\}$.
		\item Compute the conjugacy classes of $\Aut((X,L,\mathbf{R}))$.
		\item For each conjugacy class of $\Aut((X,L,\mathbf{R}))$ choose a~representative $\pi_1$ and do (5)--(7).
		\item Compute the conjugacy classes of $C_{\pi_1}(\Aut((X,L,\mathbf{R})))$.
		\item For each conjugacy class of $C_{\pi_1}(\Aut((X,L,\mathbf{R})))$ choose a~representative $\pi_2$ and do (7).
		\item If $\pi_1\pi_2$ commutes with the set  $\{L_x,\mathbf{R}_x\mid x\in X\}$ then return the $(\pi_1,\pi_2)$-isotope.
	\end{enumerate}

\begin{example}
Let $X=\{a,b,c,d\}$ be a set with
\begin{align*}
	L_a&=L_b=\mathbf{R}_a=\mathbf{R}_b=(c,d),\\
	L_c&=L_d=\mathbf{R}_c=\mathbf{R}_d=(a,b).
\end{align*}
Then $(X,L,\mathbf{R})$ is an involutive $2$-reductive square-free solution. Its automorphism group is $\Aut(X)=\langle (a,b),(a,c,b,d)\rangle$, a group of order 8 with 5 conjugacy classes. Let us go through all of them:
\begin{itemize}
	\item If $\pi_1=\mathrm{id}$ or $\pi_1=(a,b)(c,d)$ then $C_{\pi_1}(\Aut(X))=\Aut(X)$ and there are five conjugacy classes. Only three of them commute with
	both $(a,b)$ and $(c,d)$ and therefore $\pi_2\in\{\mathrm{id},(a,b),(a,b)(c,d)\}$.
	\item If $\pi_1=(a,b)$ then $C_{\pi_1}(\Aut(X))=\langle (a,b),(c,d)\rangle$. This abelian subgroup contains both $(a,b)$ and $(c,d)$ and hence $\pi_2\in\{\mathrm{id},(a,b),(c,d),(a,b)(c,d)\}$.
	\item If $\pi_1=(a,c,b,d)$ then $C_{\pi_1}(\Aut(X))=\langle (a,c,b,d)\rangle$. This cyclic subgroup has four conjugacy classes and
	$|\pi_1 C_{\pi_1}(\Aut(X))\cap \langle(a,b),(c,d)\rangle|=2$. Hence
	$\pi_2\in\{(a,c,b,d),(a,d,b,c)\}$.
	\item If $\pi_1=(a,c)(b,d)$ then $C_{\pi_1}(\Aut(X))=\langle (a,c)(b,d),(a,d)(b,c)\rangle$. This subgroup is abelian and again
	$|\pi_1 C_{\pi_1}(\Aut(X))\cap \langle(a,b),(c,d)\rangle|=2$. Hence
	$\pi_2\in\{(a,c)(b,d),(a,d)(b,c)\}$.
\end{itemize}
This computation shows that there exist, up to isomorphism, 14  solutions of multipermutation level~2 isotopic to $(X,L,\mathbf{R})$, among which 5 are involutive. By the way, 4 of these new solutions are
indecomposable, among which 2 are involutive.
\end{example}

Finally, using the algorithm described above we can straightforwardly calculate all $2$-per\-mu\-ta\-tional solutions up to size $6$. In Table \ref{Fig:count_solution}, we compare the numbers of isomorphism classes of all $2$-permutational and $2$-reductive
solutions. The numbers of $2$-reductive solutions in Table~\ref{Fig:count_solution} do not agree with \cite[Example 3.14]{JP23a} where the numbers were calculated mistakenly. For instance, in the case of size $3$ we forgot the constellation $(\mathbb{Z}_3,(1),(2))$.

\begin{table}[h!]
\begin{small}
$$\begin{array}{|r|rrrrrr|}\hline
n                       & 1& 2& 3& 4&  5&  6\\\hline
\text{$2$-permutational  solutions}    & 1& 4& 20& 219&  3113&  88604 \\ \hline
\text{$2$-reductive  solutions}    & 1& 4& 20& 207& 3061 &  88304 \\ \hline
\text{square-free $2$-reductive  solutions}    & 1& 1 & 4& 20& 183 & 2513 \\ \hline
\text{involutive solutions}    & 1& 2& 5& 23&  88&  595
\\ \hline
\text{$2$-permutational involutive solutions}    & 1& 2& 5& 19&  70&  359\\\hline
\text{$2$-reductive involutive solutions}    & 1& 2& 5& 17&  65&  323 \\ \hline
\text{square-free $2$-reductive invol. sol.}    & 1& 1& 2& 5&  15&  55 \\ \hline
\end{array}$$
\end{small}
\caption{The number of $2$-permutational solutions of size $n$, up to isomorphism.}
\label{Fig:count_solution}
\end{table}

As we can see, most 2-permutational solutions are 2-reductive. It is easy to check whether a 2-permutational solution is 2-reductive or not; in particular, in our construction we obtain a 2-reductive solution if and only if a permutation $\pi_1$ commutes with $\{\sigma_x,\tau_x\mid x\in X\}$.

\begin{prop}\label{prop:fivecond}
Let $(X,\sigma,\tau)$ be a $2$-permutational solution. 
Then the following conditions are equivalent:
\begin{enumerate}
	\item[(i)] $(X,\sigma,\tau)$ is $2$-reductive,
	\item[(ii)] $(X,\sigma,\tau)$ is distributive,
	\item[(iii)] $\{\sigma_x,\tau_x\mid x\in X\}\subseteq\Aut((X,\sigma,\tau))$,
	\item[(iv)] the set $\{\sigma_x,\tau_x\mid x\in X\}\cup\{U,T\}$ generates an abelian permutation group,
	\item[(v)] the pairs $(\sigma_x,\tau_x^{-1})$, for $x\in X$, and $(U^{-1},T)$ generate an abelian group,
	\item[(vi)] for all $x\in X$, $U\sigma_x=\sigma_xU$ and $T\tau_x=\tau_xT$.
\end{enumerate}
\end{prop}

\begin{proof}
(i)$\Rightarrow$(ii) follows by \cite[Theorems 5.5 and 5.7]{JP23a},
(ii)$\Rightarrow$(iii) by \cite[Proposition 2.13]{JPZ20}, (ii)$\Rightarrow$(i) by Lemma 3.5 and (iii)$\Rightarrow$(ii) is evident. 
\\
(i)+(ii)+(iii)$\Rightarrow$(iv) According to \cite[Theorem 4.6]{JPZ20}, for all $x,y\in X$, $\sigma_x\tau_y=\tau_y\sigma_x$. Hence
\begin{align*}
	\sigma_x\sigma_y&\stackrel{\eqref{birack:1}}=\sigma_{\sigma_x(y)}\sigma_{\tau_y(x)}
	\stackrel{\eqref{eq:per1}}=\sigma_{\sigma_y(y)}\sigma_{\tau_y(x)}\stackrel{\eqref{isomorphism}}=
\sigma_y\sigma_y\sigma_y^{-1}\tau_y\sigma_x\tau_y^{-1}=\sigma_y\sigma_x,\\
U\sigma_x(y)&=\sigma_{\sigma_x(y)}^{-1}\sigma_x(y)\stackrel{\eqref{isomorphism}}=\sigma_x\sigma_y^{-1}(y)=\sigma_xU(y)\quad \text{ and}\\
T\sigma_x(y)&=\tau^{-1}_{\sigma_x(y)}\sigma_x(y)\stackrel{\eqref{isomorphism}}=\sigma_x\tau^{-1}_y(y)=\sigma_xT(y).
\end{align*}
Analogously we can show that $U\tau_x=\tau_x U$ and $T\tau_x=\tau_x T$.
 \\
 (iv)$\Rightarrow$(v)$\Rightarrow$(vi) is evident. \\
 (vi)$\Rightarrow$(i)
 By Lemma 5.2, the bijections $U$ and $T$ are automorphisms of $(X,\sigma,\tau)$. Hence
 \begin{align*}
 \sigma_{\sigma^{-1}_x(y)}&\stackrel{\eqref{birack:1}}=\sigma_{\sigma^{-1}_y(y)}=\sigma_{U(y)}=U\sigma_yU^{-1}=\sigma_y\quad \text{ and}\\
 \sigma_{\tau_x(y)}&\stackrel{\eqref{eq:sigmasigmainverse}}=\sigma_{\sigma^{-1}_y(y)}=\sigma_{U(y)}=U\sigma_yU^{-1}=\sigma_y.
 \end{align*}
Similarly we obtain  $\tau_{\tau_x^{-1}(y)}=\tau_y$ and $\tau_{\sigma_x(y)}=\tau_y$.
\end{proof}

The assumption of multipermutation level~$2$ in Proposition \ref{prop:fivecond} is necessary  as the following example shows.

\begin{example}
Let $X=\{a,b,c,d\}$ and $\sigma_a=\tau_a=(a,c)$, $\sigma_c=\tau_c=(a,c)(b,d)$ and $\sigma_b=\sigma_d=\tau_b=\tau_d=\mathrm{id}$. Then
$(X,\sigma,\tau)$ is a solution of multipermutation level~$3$ and
$U=T=(a,c)$ commutes with both $(a,c)$ and $(a,c)(b,d)$.
Hence the solution satisfies (iv), (v) and (vi) but not (i), (ii) and (iii) of Proposition 7.2.
\end{example}

\end{document}